\documentclass[12pt]{article}
\usepackage[utf8]{inputenc}

\usepackage[margin=2.2cm]{geometry}

\usepackage{amsmath}
\usepackage{amsfonts}
\usepackage{amssymb}
\usepackage{amsthm}
\usepackage{bbm}
\usepackage{enumitem}

\usepackage{pgfplots}
\pgfplotsset{ every non boxed x axis/.append style={x axis line style=-},
     every non boxed y axis/.append style={y axis line style=-}}
\usepackage{calc}
\usepackage{wrapfig}
\usepackage{mleftright}
     
\usepackage[hyphens]{url}
\usepackage{authblk}

\newtheorem{theo}{Theorem}
\newtheorem{lemma}[theo]{Lemma}

\newtheorem{prop}[theo]{Proposition}

\theoremstyle{definition}\newtheorem{defn}[theo]{Definition}
\theoremstyle{remark}\newtheorem{remark}[theo]{Remark}
\theoremstyle{remark}

\numberwithin{theo}{section}
\numberwithin{equation}{section}

\newcommand{\R}{\mathbb{R}}
\newcommand{\N}{\mathbb{N}}

\newcommand{\MCL}{\mathcal{L}}

\newcommand{\MCE}{\mathcal{E}}
\newcommand{\MCC}{\mathcal{C}}

\newcommand{\MCH}{\mathcal{H}}

\newcommand{\MCB}{\mathcal{B}}
\newcommand{\MCM}{\mathcal{M}}
\newcommand{\MCP}{\mathcal{P}}
\newcommand{\MCS}{\mathcal{S}}
\newcommand{\MCR}{\mathcal{R}}
\newcommand{\MCA}{\mathcal{A}}
\newcommand{\MCT}{\mathcal{T}}

\newcommand{\ugamma}{\underline{\gamma}}
\newcommand{\var}{\text{var}}
\newcommand{\Leb}{\text{Leb}}
\newcommand{\id}{\text{id}}
\newcommand{\floor}[1]{\left\lfloor #1 \right\rfloor}

\newcommand{\diam}{\text{diam}}
\newcommand{\diff}{\mathop{}\!\mathrm{d}}
\newcommand{\Conv}{\mathop{}\!\text{Conv}}

\DeclareMathOperator{\interior}{int}

\title{Multifractal analysis for Markov interval maps with countably many branches}
\author{Tom Rush}
\affil{School of Mathematics, University of Bristol, Bristol, BS8 1UG, U.K.}
\affil{Email: thomas.rush@bristol.ac.uk}
\begin{document}
\date{}
\maketitle
\begin{abstract}
We study multifractal decompositions based on Birkhoff averages for sequences of functions belonging to certain classes of symbolically continuous functions. We do this for an expanding interval map with countably many branches, which we assume can be coded by a topologically mixing countable Markov shift. This generalises previous work on expanding maps with finitely many branches, and expanding maps with countably many branches where the coding is assumed to be the full shift. When the infimum of the derivative on each branch approaches infinity in the limit, we can directly generalise the results of the full countable shift case. However when this does not hold, we show that there can be different behaviour, in particular in cases where the coding has finite topological entropy.
\end{abstract}

\section{Introduction}
Given a dynamical system \((X,T)\) on a metric space \(X\) and a measurable function \(f:X \rightarrow \R\), for \(\alpha \in \R\) it is natural to consider the size (Hausdorff dimension) of the sets
\[\MCL(\alpha):=\left\{x \in X: \lim_{n \rightarrow \infty} \frac{1}{n} \sum_{i=0}^{n-1} f(T^i(x))=\alpha  \right\}.\]
When \(X\) is compact and both \(T\) and \(f\) are continuous, \(\MCL(\alpha)\) is non-empty if and only if there exists a \(T\)-invariant measure \(\mu\) such that \(\int f \diff \mu=\alpha \). Moreover, if \((X,T)\) is expanding, the space of invariant measures is typically very rich and it is often possible to express the dimension of \(\MCL(\alpha)\) as a conditional variational principle in terms of the entropies and Lypunov exponents of such measures \(\mu\). Notice that one can similarly define corresponding sets \(\MCL(\underline{\alpha})\) for finitely or countably many functions \(f_i:X \rightarrow \R\). For compact expanding dynamical systems, there is usually little difficultly in extending this result from one function to finitely or countably many functions. 

Related results have been attained in cases when \(X\) is not compact. However, without compactness, the situation is more complicated and interesting new behaviour has been found. For example, it is possible for \(\MCL(\alpha)\) to be non-empty and have positive dimension but not support any invariant measures. In this paper, we consider a dynamical system on the unit interval with a countable number of expanding \(C^1\) branches which we assume can be coded by a topologically mixing countable Markov shift (CMS). This generalises work by A. Fan, T. Jordan, L. Liao and M. Rams in their paper Multifractal Analysis for Expanding Interval Maps with Infinitely Many Branches \cite{FJLR15} where they considered the problem in the specific case where the CMS is the full shift. While much of the theory is analogous and can be seen as a direct generalisation, we also find new behaviour, in particular when the CMS has finite topological entropy. 

Interest in problems of this type goes back as far as 1934 when A. S. Besicovitch considered the Hausdorff dimension of the set of points in the unit interval whose base 2 expansion have digits with given frequencies \cite{Bes34} (see also, \cite{Kni34}). This is equivalent to finding the Hausdorff dimension of the sets \(\MCL(\alpha)\) in the case where \(X\) is the unit interval, \(T\) is doubling map and \(f\) is the characteristic function on \([0,1/2)\). In 1948, Eggleston generalised this to the base \(N\) case \cite{Egg49}, and this was further extended in papers including \cite{BSa01}, \cite{Caj81}, \cite{Dur97}, \cite{Oli98}, \cite{Oli00}, \cite{Ols02}, \cite{Ols03b}, \cite{OW03}, \cite{PS07} and \cite{Vol58}. 

In the paper Recurrence, Dimension and Entropy \cite{FFW01}, A. Fan, D. Feng and J. Wu considered the problem with a finite number of continuous functions \(f_i\) on a topologically mixing sub-shift of finite type. Related problems were also studied in the papers \cite{BSc00}, \cite{BSS02a}, \cite{BSS02b}, \cite{FF00}, \cite{FLW02}, \cite{Oli99}, \cite{Ols03a} \cite{OW07}, \cite{PW01} and \cite{Tem01}. The most fundamental application in this setting is to consider the size of the set of points with digits of given frequency. In our setting, we can consider the analogous problem relating to the set of points whose orbits occupy each branch with given frequency, that is, the set of points whose codings have digits with given frequency. As our system is coded by a countable Markov shift, there will be points whose frequency of digits sum below one. In \cite{FJLR15}, in the specific case where the coding is the full countable shift, they showed that there is some value \(s_{\infty}\), depending on the map \(T\), such that when the probabilities sum below one the dimension of the corresponding frequency sets is \(s_{\infty}\). This behaviour was first found in \cite{FLM10} for the Gauss map \(G:(0,1] \rightarrow (0,1]\) defined by \(G(x)=1/x \mod 1\). We show in Theorem \ref{theo:uiexact} that this holds in the more general setting of when the infimum of the derivative on each branch approaches infinity in the limit. However, in Theorem \ref{theo:uiexactbounded}, we find that in some instances where the derivative has a uniform bound, the dimension of these sets may vary. In this case, the CMS has finite topological entropy and the quantity \(\delta_{\infty}\), \textit{the entropy at infinity}, plays a significant role. 

A convenience in the full shift case is that one is able to approximate invariant measures using Bernoulli measures. For sub-shifts of finite type, there is a bound for the number of steps it takes to get from one digit to another on the shift space, so the argument using Bernoulli measures can be adapted. This is no longer the case for general countable Markov shifts, which causes additional complications in the analysis. We are able to work around this problem using recent work by G. Iommi, M. Todd and A. Velozo on the space of invariant measures for countable Markov shifts and the properties in the limit of sequences of these measures \cite{IV19}, \cite{ITV19}.

\subsection*{Acknowledgements}
I would like to thank my PhD supervisor, Thomas Jordan, for valuable discussions, his careful reading of multiple drafts of this paper, and his subsequent advice. I would also like to thank Mike Todd, Godofredo Iommi, the referee, and the associate editor for their helpful comments and suggestions. This work was supported by an EPSRC DTP at the University of Bristol, studentship 2278542.

\section{Setting and Results}
Let \(\{I_i\}_{i \in \N}\) be a countable collection of disjoint subintervals of \([0,1]\) which satisfy \(\cup_{i \in \N} I_i=\cup_{i \in \N} \overline{I_i}\). Let \(T_i:\overline{I_i} \rightarrow [0,1]\) be an injective \(C^1\) map such that \(|T'_i(x)| \geq \zeta>1\). By this we mean that \(T_i\) can be extended to a \(C^1\) diffeomorphism from an open neighbourhood of \(\overline{I_i}\) to an open neighbourhood of \(\overline{T_i(I_i)}\) which maps \(\overline{I_i}\) to \(\overline{T_i(I_i)}\). We define the map \(T: \cup_{i \in \N} \overline{I_i} \rightarrow [0,1]\) by \(T(x)=T_i(x)\) for all \(x \in I_i\) and adopt the convention that \(T'(x)=T'_i(x)\) for all \(x \in I_i\). We also assume that \(\log|T'|\) has variations uniformly tending to 0 (see Definition \ref{def:variationstendinguniformly}).

We assume that \(\interior {T_i(I_i)} \cap \interior{I_j}\) is equal to \(\interior{I_j}\) or the empty set for all \(i, j \in \N\), where \(\interior\) denotes the interior. Let \((\Sigma,\sigma)\) be the countable Markov shift with transition matrix \(A_{ij}=1\) if and only if \(\interior{T_i(I_i)} \cap \interior{I_j}=\interior{I_j}\). Throughout this paper we assume that this coding \((\Sigma, \sigma)\) is topologically mixing (see Section \ref{subsect:CMSandbi}). Consider the natural projection \(\Pi: \Sigma \rightarrow [0,1]\) defined by
\[\Pi(\underline{i})=\lim_{n \rightarrow \infty} T^{-1}_{i_1} \circ \ldots. \circ T_{i_n}^{-1}([0,1]).\]
Let 
\[\Lambda=\Pi(\Sigma).\]
Then \((\Lambda,T)\) defines a dynamical system.

We denote
\[E:=\{ x \in \Lambda: \# \Pi^{-1}(x)\geq 2\}\]
to be the set of points without a unique coding and note that \(\cup_{n=0}^{\infty} T^{-n}E\) is at most countable, so for any set \(\Omega \subset \Lambda\) we have that \(\dim \Omega=\dim \Omega \setminus \cup_{n=0}^{\infty} T^{-n}E\). We assume that \(\Pi(\omega_1,\omega_2, \ldots) \in I_{\omega_1}\) for every \(\omega \in \Sigma \setminus \Pi^{-1}(E)\), so that \(T(\Pi(\omega))=T_{\omega_1}(\Pi(\omega))=\Pi(\sigma \omega))\) for these \(\omega\). This can be achieved by modifying the endpoints of the \(\{I_i\}_{i \in \N}\) where necessary. We will also assume that there are no periodic points contained in \(E\). 

Let \(\MCM(\Lambda,T)\) be the set of \(T\)-invariant probability measures on \(\Lambda\). Since \(E\) is at most countable and does not contain any periodic points, it does not support any invariant measures. It follows that \(\Pi\) gives a bijection between the set of \(T\)-invariant measures and the set of shift invariant measures \(\MCM(\Sigma,\sigma)\). For \(\mu \in \MCM(\Lambda,T)\), let \(\lambda_{\mu}:=\int \log|T'| \diff \mu \) be the Lyapunov exponent of \(\mu\) and let \(h_{\mu}\) the entropy of \(\mu\) with respect to \(T\) (see Section \ref{subsec:entropy}). We also define \(\MCM_{\MCE}(\Lambda,T)\) and \(\MCM_{\MCE}(\Sigma,\sigma) \) to be the subsets of \(\MCM(\Lambda,T)\) and \(\MCM(\Sigma,T)\), respectively, consisting of the ergodic measures. It is easy to see that \(\Pi\) also gives a bijection between \(\MCM_{\MCE}(\Lambda,T)\) and \(\MCM_{\MCE}(\Sigma,\sigma)\). 

For a sequence of functions \(\phi_i:\Lambda \rightarrow \R\) with variations uniformly tending to 0 (Definition \ref{def:variationstendinguniformly}), we will study the possible limit points in \(\R^{\N}\) of the Birkhoff average sequences \( (A_n \phi_i(x))_{n \in \N} \), where
\[A_n \phi(x):= \frac{1}{n} \sum_{i=0}^{n-1} \phi(T^i(x)). \]
In particular, we investigate sets of the form 
    \[\Lambda(\underline{\gamma}):=\{ x \in \Lambda \setminus \cup_{j=0}^{\infty} T^{-j}E:\lim_{n \rightarrow \infty} A_n \phi_i(x)=\gamma_i \text{ for all } i \in \N\}, \: \underline{\gamma} \in \R^{\N} .\]
The following sets will be used to describe the possible limits of the Birkhoff averages. Let
\[Z_0:=  \left\{ \underline{\gamma} \in \R^{\N}: \exists \mu \in \MCM(\Lambda,T), \int \phi_i \diff \mu=\gamma_i, \forall i \in \N \right\}\]
and let \(Z\) be the closure of \(Z_0\) in the pointwise limit topology, that is
\[Z:= \left\{ \underline{\gamma} \in \R^{\N}: \forall \varepsilon>0, \forall k \in \N, \exists \mu \in \MCM(\Lambda, T), \forall i \leq k, \left|\int \phi_i \diff \mu-\gamma_i \right|<\varepsilon \right\}. \]

The following set will be of importance in this paper. Let \[\MCR:= \bigcup_{q=1} \{\omega \in \Sigma: \omega_i=q \text{ for infinitely many } i \in \N \}\]
be the \textit{recurrent set}. We will also use \(\MCT\) to denote the \textit{transient set} \(\Sigma \setminus \MCR\).  For a set \(\Omega \subset \Sigma\) we denote the set \(\Pi(\Omega)\setminus \cup_{n=0}^{\infty} T^{-n}E\) by \(\Lambda_\Omega\) and \(\Pi(\Omega) \cap \Lambda(\underline{\gamma})\) by \(\Lambda_\Omega(\underline{\gamma})\). Unfortunately, unlike in the case where the coding is the full countable shift, there may exists \(\ugamma \not\in Z\) which is the Birkhoff limit of some \(x \in \Lambda_\MCT\). For this reason, we must restrict our attention to the recurrent set \(\MCR\). We calculate the Hausdorff dimension of the sets \(\Lambda_{\MCR}(\underline{\gamma})\). For \(\ugamma \in Z\), let
\[\alpha_1(\ugamma):=\lim_{\varepsilon \rightarrow 0} \lim_{k \rightarrow \infty} \sup_{\mu \in \MCM(\Lambda,T)} \left\{\frac{h_{\mu}}{\lambda_{\mu}}:\left|\int \phi_i \diff \mu-\gamma_i \right|<\varepsilon, \forall i \leq k, \: \lambda_{\mu}<\infty \right\} \]
\[\alpha_2(\ugamma):=\lim_{\varepsilon \rightarrow 0} \lim_{k \rightarrow \infty} \sup_{\mu \in \MCM_{\MCE}(\Lambda,T)} \left\{\frac{h_{\mu}}{\lambda_{\mu}}:\left|\int \phi_i \diff \mu-\gamma_i \right|<\varepsilon, \forall i \leq k, \: \lambda_{\mu}<\infty \right\}. \]

\begin{theo}\label{theo:uiapprox}
Let \((\phi_i)_{i \in \N}\) be a sequence of functions with variations uniformly tending to 0. For \(\gamma \not\in Z\), we have \(\Lambda_{\MCR}(\underline{\gamma})=\emptyset\). For \(\gamma \in Z\), we have
\[\dim \Lambda_{\MCR}(\underline{\gamma})=\alpha_1(\ugamma)=\alpha_2(\ugamma).\]
\end{theo}

We would like to have the dimension without the limits in \(k\) and \(\varepsilon\). This is possible if we restrict the behaviour of \(|T'|\) on \(I_i\) in the limit as \(i \rightarrow \infty\), and set stricter conditions on the \(\phi_i\). For Theorem \ref{theo:uiexact} we assume (in addition to our previous assumptions) that
\begin{equation}\label{eqn:uiexactcondition}
    \inf\{|T'(x)|:x \in \overline{I_i}\setminus E\} \rightarrow \infty
\end{equation}
as \(i \rightarrow \infty\), and that the \(\phi_i\) are bounded. In this case we can extend the result of Theorem 1.2 in \cite{FJLR15}. For \(\ugamma \in Z_0\), let
\[\alpha_3(\ugamma):=\sup_{\mu \in \MCM(\Lambda,T)} \left\{\frac{h_{\mu}}{\lambda_{\mu}}:\int \phi_i \diff \mu =\gamma_i, \forall i \in \N, \lambda_{\mu}<\infty \right\}. \]
Analogously to \cite{FJLR15}, we define
\[s_{\infty}=\inf \{ t \geq 0:P(-t\log|T'|)<\infty\}, \]
where, for a function \(\phi :\Lambda \rightarrow \R\), \(P(\phi)\) is the pressure defined by
\[P(\phi):=\sup_{\mu \in \MCM(\Lambda,T)} \left\{h_{\mu}+\int \phi \diff \mu: -\int \phi \diff \mu<\infty \right\}. \]
We remark that if the topological entropy \(h_{\text{top}}(\Sigma,\sigma)\) is finite (see Section \ref{subsec:toppressure}), then \(s_{\infty}=0\).

\begin{theo}\label{theo:uiexact}
In the setting of Theorem \ref{theo:uiapprox}, let \(|T'|\) further satisfy condition (\ref{eqn:uiexactcondition}) and let the \((\phi_i)_{i \in \N}\) be bounded. For \(\ugamma \in Z_0\) we have
\[\dim \Lambda_{\MCR}(\underline{\gamma})=\max\{s_{\infty},\alpha_3(\ugamma)\}.  \]
Moreover, for \(\ugamma \in Z \setminus Z_0\) we have
\[\dim \Lambda_{\MCR}(\ugamma) =s_{\infty}. \]
\end{theo}

\begin{remark}
The methods used in \cite{FJLR15} can be modified to hold when the coding satisfies the big images and pre-images (BIP) property (see \cite[Definition 5.8]{Sar15}). Therefore, if \((\Sigma,\sigma)\) satisfies the BIP property, Theorem \ref{theo:uiapprox} and Theorem \ref{theo:uiexact} hold without taking the intersection with \(\Pi(\MCR)\). 
\end{remark}

We state our final theorem now but defer the definitions until Section \ref{section:countablemarkovshifts}. If, in addition to the conditions in Theorem \ref{theo:uiapprox}, we assume the functions \(\phi_i \in C_0(\Lambda) \) and \(\log|T'|-L \in C_0(\Lambda)\) for some \(L \geq \log \zeta\),  where \(C_0(\Lambda)\) is some class of bounded functions with variations uniformly tending to 0 which symbolically vanish at infinity, we can also get an exact-type result. We remark that \(\log|T'|\) being bounded implies that the topological entropy \(h_{\text{top}}(\Sigma,\sigma)\) is finite since \(h_{\mu}\leq \lambda_{\mu}\) for all \(\mu \in \MCM(\Lambda,T)\) (see Lemma \ref{lem:entropylessthanlyapunov} and Remark \ref{remark:entropylessthanlypunov}) and entropy is preserved under \(\Pi\). As the functions \(\phi_i\) are in \(C_0(\Lambda)\), for \(\ugamma \in Z \setminus \{0\} \) we have \(\Lambda(\ugamma) =\Lambda_{\MCR}(\ugamma)\) trivially. This allows us to calculate the dimension of the sets \(\Lambda(\ugamma)\) without taking any intersection. With these conditions, we show that \(Z\) can be written as
\[Z=\left\{ \underline{\gamma} \in \R^{\N}: \exists \mu \in \MCM_{\leq 1}(\Lambda,T), \forall i \in \N, \int \phi_i \diff \mu=\gamma_i \right\},\]
where \(\MCM_{\leq 1}(\Lambda,T)\) is the set of \(T\)-invariant sub-probability measures, that is, the set of \(T\)-invariant measures \(\mu\) with \(|\mu|:=\mu(\Lambda)<1\). For \(\ugamma \in Z\), let 
\[\alpha_4(\ugamma):= \sup_{\mu \in \MCM_{\leq 1}(\Lambda,T)} \left\{\frac{|\mu|h_{\frac{\mu}{|\mu|}}+(1-|\mu|)\delta_{\infty}}{|\mu|\lambda_{\frac{\mu}{|\mu|}}+(1-|\mu|)L}:\int \phi_i \diff \mu=\gamma_i  \right\},\]
where \(\delta_{\infty}\) is the entropy at infinity of \((\Sigma,\sigma)\) (see Definition \ref{def:entropyatinfinity}). When \(\mu\) is the zero measure, the quantity in the brackets is to be interpreted as \(\delta_{\infty}/L\).

\begin{theo}\label{theo:uiexactbounded}
In the setting of Theorem \ref{theo:uiapprox}, further assume that \((\phi_i)_{i \in \N} \subset C_0(\Lambda)\) and suppose that \(\log |T'|-L \in C_0(\Lambda)\) for some \(L \geq \log \zeta\). Then for \(\ugamma \in Z \setminus \{\underline{0}\}\) we have 
\[\dim \Lambda(\underline{\gamma})=\alpha_4(\ugamma).\]
Moreover, \(\underline{0} \in Z\) and satisfies
\[\dim \Lambda(\underline{0})=\max\{\alpha_4(\underline{0}),\dim \Lambda_{\MCT}\}. \]
\end{theo}

\begin{remark}
If instead we have \((\phi_i-L_i) \subset C_0(\Lambda)\) for some sequence \((L_i)_{i \in \N} \subset \R^{\N}\), then we can get the dimensions of \(\Lambda(\ugamma)\) by applying Theorem \ref{theo:uiexactbounded} to the functions \(\phi_i-L_i \in C_0(\Lambda)\).
\end{remark}

\begin{remark}\label{remark:noperiodic}
When \(E\) does contain periodic points, Theorems \ref{theo:uiapprox}, \ref{theo:uiexact} and \ref{theo:uiexactbounded} with hold more generally if \(Z_0\), \(Z\), \(s_{\infty}\), and the \(\alpha_i(\ugamma)\) are instead defined in terms of shift invariant measures and the corresponding uniformly continuous functions on the shift space. The difference in this case is due to shift invariant measures being able to give mass to \(\Pi^{-1}(E)\), where \(\phi_i \circ \Pi\) and \(\log|T' \circ \Pi|\) can have points of discontinuity (see Section \ref{subsec:bd}). Note that only minor modifications to our proofs are needed for this more general setting since we work predominantly on the shift space. However, we do not provide the details here.
\end{remark}

In the next section we introduce some basic definitions and prove some distortion estimates. In Section \ref{section:countablemarkovshifts} we then recall some theory of countable Markov shifts with finite topological entropy. While we are not assuming \((\Sigma,\sigma)\) has finite topological entropy, the suspension space on \(\Sigma\) with roof function \(\log|T'|\) will have finite topological entropy. Hence in Section \ref{sec:susspace}, via Abramov's formula we are able to relate the quantities \(\frac{h_{\mu}}{\lambda_{\mu}}\) to the entropy of a measure on a CMS with finite topological entropy. We use this to prove two propositions: Proposition \ref{prop:entropydensity} and Proposition \ref{prop:uppersemicontinuity}. The first allows us to approximate arbitrary shift invariant measures with ergodic measures supported on finite sub-shifts. The second gives us upper semi-continuity of the map \(\mu \mapsto \frac{h_{\mu}}{\lambda_{\mu}}\) in the weak* topology when the measures satisfy \(\frac{h_{\mu}}{\lambda_{\mu}} > s_\infty \) (in the setting of Theorem \ref{theo:uiexact}). Note that this is a necessary tool for proving Theorem \ref{theo:uiexact} from Theorem \ref{theo:uiapprox}. In \cite{FJLR15} they prove this in the specific case where \((\Sigma,\sigma)\) is the full shift. However, their methods rely on the uniform structure of the shift space and cannot be used here. The following three sections are devoted to proving Theorem \ref{theo:uiapprox}, Theorem \ref{theo:uiexact} and Theorem \ref{theo:uiexactbounded} respectively. Finally, in Section \ref{sec:applications} we discuss some applications. In particular, to the frequency of digits case and the map \(F_{\lambda}:(0,1] \rightarrow (0,1]\), defined for \(\lambda \in (0,1)\) by 
\[F_{\lambda}(x):=
  \begin{cases}
    \frac{x-\lambda}{1-\lambda}, & \text{for } x \in I_1 \\
    \frac{x-\lambda^n}{\lambda(1-\lambda)}, & \text{for } x \in I_n, n \geq 2,
  \end{cases}\]
which was studied in \cite{SV97}, \cite{BT12}, \cite{BT15} and \cite{IJT17}. We finish by discussing some cases when \(\alpha_4(\ugamma)\) can instead be written as a supremum over probability measures. 

\section{Preliminary definitions and lemmas}\label{sec:prelims}
\subsection{CMS and basic intervals}\label{subsect:CMSandbi}
Let \(\N^{\N}\) be the set
\[\N^{\N}:=\{ (\omega_1,\omega_2,\omega_3,\ldots):\omega_i \in \N \}. \]
Given an \(\N \times \N\) matrix \(A\) with entries \(0\) or \(1\), consider the \textit{countable Markov shift} \( (\Sigma,\sigma) \) where 
\[\Sigma:=\{\omega \in \N^{\N}: A_{\omega_i,\omega_{i+1}}=1, \forall i \in \N\}\]
and the \textit{shift map} \(\sigma : \Sigma \rightarrow \Sigma\) is defined by
\[\sigma (\omega_1,\omega_2,\omega_3,\ldots)=(\omega_2,\omega_3,\ldots).\] 
An admissible word of length \(n\) is a string \((\omega_1,\omega_2,\ldots,\omega_n) \in \N^n\) such that \(A_{\omega_i,\omega_{i+1}}=1\) for all \(i=1,\ldots,n-1\). We denote by \(\MCC_{n}(\Sigma)\) the set of all admissible words of length \(n\). For a point \(\omega=(\omega_1,\omega_2,\ldots) \in \Sigma\) we use \(\omega|_i^j\) to denote the word \((\omega_i,\ldots,\omega_j) \in \MCC_{j-i+1}(\Sigma)\). For \((\omega_1,\ldots,\omega_n) \in \MCC_{n}(\Sigma)\) the \(n\)th level cylinder \([\omega_1,\ldots,\omega_n]\) is defined by \(\{\omega' \in \Sigma:\omega'_i=\omega_i, \forall i=1,\ldots,n\}\) and the \(n\)th level basic interval is defined by
\[C_n(\omega_1,\ldots,\omega_n):=\Conv(\Pi([\omega_1,\ldots, \omega_n]) \setminus \cup_{j=0}^{\infty} T^{-j}E),\]
where \(\Conv\) denotes the convex hull. An admissible word \(w\) is said to connect \(a,b \in \N\) if the cylinder \([a,w,b]\) is non-empty.  In this paper we always assume the coding \((\Sigma,\sigma)\) corresponding to \((\Lambda,T)\) is \textit{topologically mixing}, that is, for each pair \(a,b \in \N\) there exists an \(N \in \N\) such that for all \(n \geq N\) there is an admissible word of length \(n\) connecting \(a\) and \(b\). In Section \ref{sec:susspace}, from the coding \((\Sigma,\sigma)\) we will construct a CMS which may not be topologically mixing, but will be \textit{topologically transitive}. This is a weaker condition and means that for each pair \(a,b \in \N\) there exists an admissible word connecting \(a\) and \(b\).

We endow \(\Sigma\) with the topology generated by the cylinders.  We use the metric \(d:\Sigma \times \Sigma \rightarrow \R\) on \(\Sigma\) defined by
\[d(\omega,\omega')=
\begin{cases}
1 & \text{ if } \omega_1 \not= \omega'_1 \\
2^{-k} & \text{ if } \omega_i = \omega'_i \text{ for all } i=1,\ldots,k \text{ and } \omega_{k+1} \not= \omega'_{k+1} \\
0 & \text{ if } \omega=\omega'.
\end{cases} \]
This generates the same topology as that of the cylinders. 

\subsection{Bounded distortion}\label{subsec:bd}
For \(\phi:\Lambda \rightarrow \R\), let
\[\var_n(\phi):=\sup\{|\phi(x)-\phi(y)|:x,y \in C_n(\omega), \omega \in \Sigma\}.\]

\begin{defn}\label{def:variationstendinguniformly}
Let \(\phi:\Lambda \rightarrow \R\). We say that \(\phi\) has \textit{variations uniformly tending to} 0 if \(\var_1 \phi < \infty\) and \(\lim_{n \rightarrow \infty} \var_n (\phi)=0\).
\end{defn}

We define \(\log|T'|:\cup_{i \in \N} \overline{I_i} \rightarrow \R\) to have variations uniformly tending to 0 similarly. Note that if \(\phi\) has variations uniformly tending to 0 then there exists a uniformly continuous function, which we will denote by \(f_{\phi}\), such that \(f_{\phi}(\omega)=\phi \circ \Pi(\omega)\) for all \(\omega \in \Sigma \setminus \Pi^{-1}(E)\). Since we are assuming \(E\) does not contain any periodic points, it cannot support any invariant measures. It follows that for every \(\nu \in \MCM(\Sigma,\sigma)\) and every \(\phi\) with variations uniformly tending to 0 
\begin{equation}\label{eqn:integralequality}
    \int f_{\phi} \diff \nu= \int \phi \circ \Pi \diff \nu.
\end{equation}
We also have that for every \(\phi\) with variations uniformly tending to 0, every \(n \in \N\), and all \(x \in \Lambda \setminus \cup_{j=0}^{\infty} T^{-j} E\)
\begin{equation}\label{eqn:birkhoffaverageequality}
    A_n \phi(x)=A_n f_{\phi}(\Pi^{-1}(x)).
\end{equation} 

Given a basic interval \(C_n(\omega_1,\ldots, \omega_n)\), we define
\[M^* \phi(\omega_1,\ldots,\omega_n):=\sup_{\omega \in C_n(\omega_1,\ldots,\omega_n)} A_n \phi(\omega)\]
\[M_* \phi(\omega_1,\ldots,\omega_n):=\inf_{\omega \in C_n(\omega_1,\ldots,\omega_n)} A_n \phi(\omega).\]

\begin{lemma}\label{lem:variationsgoingtozero}\cite[Lemma 2.2]
{FJLR15}
Let \(\phi: \Lambda \rightarrow \R\) have variations uniformly tending to 0. Then 
\[\lim_{n \rightarrow \infty} \sup_{(\omega_1,\ldots,\omega_n)\in \MCC_n(\Sigma)} M^* \phi(\omega_1,\ldots,\omega_n)-M_* \phi(\omega_1,\ldots,\omega_n)=0.\]
\end{lemma} 

\begin{proof}
This follows straightforwardly as for any \(n \in \N\) and \((\omega_1,\ldots, \omega_n) \in \MCC_{n}(\Sigma) \)
\[|M^* \phi(\omega_1,\ldots,\omega_n)-M^* \phi(\omega_1,\ldots,\omega_n)| \leq \frac{1}{n} \sum^{n}_{j=1} \var_j \phi \xrightarrow[n \rightarrow \infty]{} 0.\]
\end{proof}

We adapt Lemma 2.3 in \cite{FJLR15} to our setting. Note that we have an additional term since we are allowing \(\diam(T(I_{i}))\) to be less than 1.

\begin{lemma}\label{lem:diameterfunction}
There exists a positive sequence \(\varepsilon(n)\) converging to 0 such that for any \(\omega \in \Sigma \setminus \Pi^{-1}(\cup_{j=0}^{\infty} T^{-j} E)\) we have
\[\left|\frac{\log(\emph{diam}(C_n(\omega)))}{n}- A_n(-\emph{log}|T' \circ \Pi(\omega)|) \right| \leq \frac{- \log(\emph{diam}(T(C_1(\omega_n))))}{n} + \varepsilon(n). \]
\end{lemma}

\begin{proof}
By the Mean Value Theorem we have 
\[\log (\diam( C_n(\omega)))-n A_n(-\log|T'(x)|)=\log( \diam(T^n( C_n(\omega)))\]
for some \(x \in C_n(\omega)\). Hence,
\[\left|\log( \diam( C_n(\omega)))-n A_n(-\log|T'(x) \right|) \leq - \log(\diam(T(C_1(\omega_n)))).\] 
We can then apply Lemma \ref{lem:variationsgoingtozero} to \(\log|T'|\) since we are assuming it has variations uniformly tending to 0.
\end{proof}

\subsection{Entropy of invariant measures}\label{subsec:entropy}
We briefly recall the definition of the entropy of an invariant probability measure (for more details, see \cite[Chapter 4]{Wal81}). Let \((Y,\MCB, f, \mu)\) be a probability-preserving transformation. A partition \(\beta\) is a finite or countable collection of subsets \(\xi_i \in \MCB\) such that
\(\xi_i \cap \xi_j = \emptyset\) if \(i \not= j\) and \(\bigcup_i \xi_i=Y\). The \textit{entropy of the partition} is defined to be
\[H_{\mu}(\beta)= - \sum_i \mu(\xi_i) \log \mu(\xi_i) \] 
where \(0 \log 0 :=0\). Note it is possible for \(H_{\mu}(\xi)\) to be infinite. For a partition \(\beta\), we define \(f^{-1} (\beta):= \{ f^{-1}(\xi_i):\xi_i \in \beta\}\). Then \(f^{-1}(\beta)\) is also a partition. Furthermore, for two partitions \(\beta, \beta'\) we define the join \(\beta \wedge \beta'\) to be the set \(\{ \xi_i \cap \xi'_j:\xi_i \in \beta, \xi'_j \in \beta'\}\). Again, the set \(\beta \wedge \beta'\) is also a partition. The \textit{entropy of \(\mu\) with respect to \(\beta\)} is then defined to be
\[h_{\mu}(\beta):= \lim_{n \rightarrow \infty} \frac{1}{n} H_{\mu} \left(\wedge_{i=0}^{n-1} f^{-1}(\beta) \right).\] 
Finally, the \textit{entropy of \(\mu\)} is defined to be 
\[ h_{\mu}(\sigma):=\sup_{\beta}\{ h_{\mu}(\beta): \beta \text{ is a partition with } H_{\mu}(\beta)<\infty \}.\]

We have the following lemma bounding the entropy of measures \(\mu \in \MCM_{\MCE}(\Lambda,T)\) by their Lyapunov exponents. The proof is standard and likely known. However, as there are some adjustments needed to prove it in the setting we are working in, we include it here for completeness.

\begin{lemma}\label{lem:entropylessthanlyapunov}
For any \(\mu \in \MCM_{\MCE}(\Lambda,T)\), we have \(h_{\mu} \leq \lambda_{\mu} \).
\end{lemma}

\begin{proof}
Clearly this holds if \(\lambda_{\mu}= \infty \), so assume \( \lambda_{\mu}<\infty.\) We prove the case when \(h_{\mu}<\infty\); the case when \(h_{\mu}=\infty\) can be proved similarly. Let \(\varepsilon>0\).  Given \(\mu \in \MCM_{\MCE}(\Lambda,T)\), let \(\nu \in \MCM_{\MCE}(\Sigma,\sigma)\) be such that \(\mu=\Pi^* \nu \) and let \(k \in \N\) be such that \(\nu(\cup_{i \leq k} [i])>1-\varepsilon \). By the Shannon-McMillan-Breiman theorem, Birkhoff's ergodic theorem, Egorov's theorem and Lemma \ref{lem:diameterfunction}, there exists \(\Omega \subset \Sigma \) with \(\nu(\Omega)>1-\varepsilon\) and \(N \in \N\) such that for all \(\omega \in \Omega\) and all \(n \geq N\)
\begin{equation*}
    \nu([\omega_1,\ldots, \omega_n])\leq e^{-n(h_{\mu}-\varepsilon)}
\end{equation*}
\begin{equation*}
    A_n(-\log|T' \circ \Pi(\omega)|)\geq-(\lambda_{\mu}+\varepsilon)
\end{equation*}
\begin{equation*}
    \sup_{\omega \in \Omega, \omega_n \leq k} \left|\frac{\text{log}(\diam(C_n(\omega)))}{n}- A_n(-\log|T' \circ \Pi(\omega)|) \right| \leq \varepsilon.
\end{equation*}
For any \(n \geq N\), let
\[\MCA_n:=\{(\omega_1, \ldots, \omega_n) \in \MCC_n(\Sigma):[\omega_1, \ldots, \omega_n] \cap \{ \omega \in \Omega: \omega_n \leq k  \} \not= \emptyset  \}. \] 
We have 
\[1-2\varepsilon<\nu(\{ \omega \in \Omega: \omega_n \leq k  \}) \leq  \# \MCA_n   e^{-n(h_{\mu}-\varepsilon)} \]
and
\[1 \geq \sum_{(\omega_1, \ldots, \omega_n) \in \MCA_n} \diam(C_n(\omega_1, \ldots, \omega_n))  \geq \# \MCA_n e^{-n(\lambda_{\mu}+2\varepsilon)}.  \]
Hence
\[ e^{n(\lambda_{\mu}+2\varepsilon)} \geq  (1-2\varepsilon) e^{n(h_{\mu}-\varepsilon)},\]
and so 
\begin{align*}
    \lambda_{\mu} &\geq \frac{1}{n} \log(1-2\varepsilon)+h_{\mu}-3 \varepsilon \\
    &\geq \log(1-2\varepsilon)+h_{\mu}-3 \varepsilon.
\end{align*}
The result follows by letting \(\varepsilon \rightarrow 0\).
\end{proof}

\begin{remark}\label{remark:entropylessthanlypunov}
Lemma \ref{lem:entropylessthanlyapunov} is sufficient for our use in the proof of Theorem \ref{theo:uiapprox}. However, it then follows from Theorem \ref{theo:uiapprox} that the inequality \(h_{\mu} \leq \lambda_{\mu}\) holds for all \(\mu \in \MCM(\Lambda,T)\). 
\end{remark}

\subsection{Topological pressure and topological entropy}\label{subsec:toppressure}

Recall that for a function \(\phi :\Lambda \rightarrow \R\), the pressure \(P(\phi)\) is defined by
\[P(\phi):=\sup_{\mu \in \MCM(\Lambda,T)} \left\{h_{\mu}+\int \phi \diff \mu: -\int \phi \diff \mu<\infty \right\}. \]
While this definition of pressure is sufficient for our purposes, we remark that when \(\phi\) has summable variations, that is \(\sum_{n=1}^{\infty} \var_n(\phi)<\infty \), this can be alternatively stated (see \cite[Theorem 3]{Sar99} and \cite[Theorem 2.10]{IJT15}) as 
\begin{equation}\label{eqn:topologicalpressure}
    P_{\text{top}}(\phi):=\lim_{n \rightarrow \infty} \frac{1}{n} \log \sum_{\substack{\sigma^n \omega=\omega}} e^{S_n \phi \circ \Pi(\omega)} 1_{[a]}(\omega), 
\end{equation}
where the value does not depend on the \(a \in \N\) chosen. This is equivalently the \textit{Gurevich pressure} of \(\phi \circ \Pi\) on \((\Sigma,\sigma)\). This was defined by Sarig in \cite{Sar99} based on work by Gurevich \cite{Gur69}. The Gurevich pressure of the zero function is of special importance. We call this the \textit{topological entropy} of \((\Sigma,\sigma)\) and denote it by \(h_{\text{top}}(\Sigma,\sigma)\). Explicitly, 
\[h_{\text{top}}(\Sigma,\sigma):= \lim_{n \rightarrow \infty} \frac{1}{n} \sum_{\sigma^n \omega= \omega} 1_{[a]}(\omega).\]
If \((\Sigma,\sigma)\) is topologically transitive but not topologically mixing then the limit is replaced by a limsup. In either case, this satisfies the variational principle 
\begin{align}\label{eqn:hvariationalprinciple}
    h_{\text{top}}(\Sigma,\sigma) &=\sup\{ h_\nu : \nu \in \MCM(\Sigma, \sigma)\} \nonumber \\
    &= \sup\{ h_\nu : \nu \in \MCM_{\MCE}(\Sigma, \sigma)\},
\end{align}
where the second equality follows from \cite[Theorem 2]{Sar99} and \cite[Corollary 8.6.1]{Wal81}. 

\section{Countable Markov shifts with finite topological entropy}\label{section:countablemarkovshifts}
In this section we recall some definitions and theory for countable Markov shifts with finite topological entropy. For a more thorough account we refer the reader to \cite{IV19} and \cite{ITV19}. We emphasise that we are not in general assuming the shift space \((\Sigma,\sigma)\) corresponding to \((\Lambda,T)\) has finite topological entropy. However, in Section \ref{sec:susspace} we show that the quantities \(\frac{h_{\mu}}{\lambda_{\mu}}\) can be related to the entropies of measures on a topologically transitive (but not necessarily topologically mixing) CMS with finite topological entropy. Furthermore, under the conditions of Theorem \ref{theo:uiexactbounded}, the topological entropy of \((\Sigma,\sigma)\) corresponding to \((\Lambda,T)\) will indeed be finite, so the theory outlined here will be useful for us once again. Throughout this section, to account for the fact that the CMS in Section \ref{sec:susspace} may not be topologically mixing, by \((\Sigma,\sigma)\) we mean a topologically transitive CMS with finite topological entropy.

\subsection{The space of invariant measures}
Recall we denote by \( \MCM(\Sigma,\sigma)\) the set of all \(\sigma\)-invariant probability measures on \(\Sigma\) and \(\MCM_{\MCE}(\Sigma,\sigma)\) the subset of \( \MCM(\Sigma,\sigma)\) consisting of the ergodic measures. The following proposition was proved in \cite{ITV19} (Theorem 8.7). While they state in the theorem that the measures \(\nu_n\) may be taken to be compactly supported, they prove the stronger result stated here. Notice that as \(\log|T' \circ \Pi| \) is not necessarily bounded we cannot also conclude that \(\lim_{n \rightarrow \infty} \lambda_{\nu_n}=\lambda_{\nu} \). However, in Proposition \ref{prop:entropydensity} we extend this result allowing us to assume this also holds.

\begin{prop}\label{prop:entropydensitymike}
    Let \(\nu \in \MCM(\Sigma,\sigma)\), then there exists a sequence of ergodic measures \(\nu_n \in \MCM(\Sigma,\sigma)\) such that \(\nu_n\) converges to \(\nu\) in the weak* topology and  \(\lim_{n \rightarrow \infty} h_{\nu_n} =h_{\nu}\). It is moreover possible to choose the \(\nu_n\) such that they are supported on finitely many symbols, that is, supported on \(\{1,\ldots,k_n\}^{\N}\), respectively, for some sequence \((k_n)_{n \in \N} \subset \N\).
\end{prop}

It is well known that \(\MCM(\Sigma,\sigma)\) is not compact in the weak* topology as mass may be lost in the limit. Accordingly, let \(\MCM_{\leq 1}(\Sigma,\sigma)\) be the set of sub-probability \(\sigma\)-invariant measures. This is defined to be the set of \(\sigma\)-invariant measures such that \(|\nu| \leq 1 \) where \(|\nu|:=\nu(\Sigma)\). Given an enumeration \(C_i\) of the cylinders of \(\Sigma\), we define the metric on \(\MCM_{\leq 1}(\Sigma,\sigma)\) by
\[\rho(\nu, \eta) = \sum_{i=1}^{\infty} \frac{1}{2^n} |\nu(C_i)-\eta(C_i)|\]
The topology induced by this metric is called the \text{topology of convergence on cylinders}. We say a sequence of measures \(\nu_n \in \MCM_{\leq 1}(\Sigma,\sigma)\) \textit{converges to} \(\nu \in \MCM_{\leq 1}(\Sigma,\sigma)\) \textit{on cylinders} if for every cylinder \(C \subset \Sigma\)
\[\lim_{n \rightarrow \infty} \nu_n(C)=\nu(C).\] 
Clearly a sequence \(\nu_n\) converges to a measure \(\nu\) on cylinders if and only if it converges to \(\nu\) in the topology of convergence on cylinders. It was shown in \cite{IV19} (Theorem 1.2) that \(\MCM_{\leq 1}(\Sigma,\sigma)\) endowed with this topology is compact. Moreover, weak* convergence and convergence on cylinders are equivalent when there is no loss of mass (see \cite[Lemma 3.17]{IV19}). We can further characterise this topology in terms of test functions as follows. If \(C\) is a cylinder of length \(m\), denote by 
\[C(\geq n):=\left\{ \omega \in C: \sigma^m(\omega) \in \bigcup_{k \geq n} [k] \right\}. \] 
For a non-empty set \(\Omega \subset \Sigma\) we define 
\[\var^{\Omega}(f):=\sup \{ |f(\omega)-f(\omega')|:\omega,\omega' \in \Omega\}. \]
We say \(f:\Sigma \rightarrow \R\) is in \(C_0 (\Sigma)\) if and only if the following four conditions hold:
\begin{enumerate}
\item f is bounded
\item f is uniformly continuous
\item \(\lim_{n \rightarrow \infty} \sup_{\omega \in [n]} |f(x)|=0\)
\item \(\lim_{n \rightarrow \infty} \var^{C( \geq n)}(f)=0\), for every cylinder \(C \subset \Sigma\). 
\end{enumerate}

\begin{defn}\label{def:C0}
We further define \(C_0(\Lambda)\) to be the subset of functions \(\phi: \Lambda \rightarrow \R \) with variations uniformly tending to 0 such that \(f_{\phi} \in C_0(\Sigma)\), where \(f_{\phi}\) is as defined as in Section \ref{subsec:bd}.
\end{defn}

The following lemma was shown in \cite{IV19} (see Lemma 3.19) and characterises the topology of convergence on cylinders in terms of test functions. We will use this lemma multiple times in the proof of Theorem \ref{theo:uiexactbounded}. 
\begin{lemma}\label{lem:convergenceoncylinderstestfunctions}
Let \(\nu,(\nu_n)_{n \in \N} \subset \MCM_{\leq 1}(\Sigma,\sigma)\), then \(\nu_n\) converges to \(\nu\) in the topology of convergence on cylinders if and only if for every \(f \in C_0(\Sigma)\)
\[\int f \diff \nu_n \rightarrow \int f \diff \nu.\]
\end{lemma}

\subsection{Entropy at infinity}
An important quantity of a CMS with finite topological entropy is the \textit{entropy at infinity}. This is a measure of how complex the system is near infinity.  

\begin{defn}[{\cite[Definition 1.2]{ITV19}}]\label{def:entropyatinfinity}
Let \( (\Sigma,\sigma)\) be a CMS. Let \(M,q \in \N\). For \(n \in \N \) let \(z_n(M,q)\) be the
number of cylinders of the form \([\omega_1, . . . , \omega_{n+2}]\), where \(\omega_1 \leq q, \omega_{n+2} \leq q\), and
\[ \# \{ i \in \{1, \ldots , n +2\} : \omega_i \leq q \} \leq \frac{n+2}{M}.\]
Define
\[ \delta_{\infty}(M,q) := \limsup_{n \rightarrow \infty} \frac{1}{n} \log z_n(M,q) \]
and
\[\delta_{\infty}(q):= \lim_{M \rightarrow \infty} \delta_{\infty}(M,q).\]
The topological entropy at infinity of \((\Sigma,\sigma)\) is defined by \(\delta_\infty := \lim_{q \rightarrow \infty} \delta_{\infty} (q)\).
\end{defn}

This too satisfies a variational principle. Theorem 1.4 in \cite{ITV19} says 
\[\delta_{\infty}=h_{\infty}, \]
where 
\[h_{\infty}:= \sup_{(\nu_n)_n \rightarrow 0} \limsup_{n \rightarrow \infty} h_{\nu_n}  \]
is the \textit{metric theoretic entropy at infinity of \((\Sigma,\sigma)\)}. Here \((\nu_n)_n \rightarrow 0\) means that the sequence of measures converges to the zero measure on cylinders and the supremum is over all such sequences \( (\nu_n)_{n \in \N} \subset \MCM(\Sigma, \sigma)\).

The following theorem was proved in \cite{ITV19}. We will use this to prove Proposition \ref{prop:uppersemicontinuity} and in the proof of the upper bound of Theorem \ref{theo:uiexactbounded}. Note that it also gives upper semi-continuity of the entropy map in the weak* topology.

\begin{theo}\label{theo:uppersemicontiuitymike}
Let \((\Sigma,\sigma)\) be a topologically transitive CMS with finite topological entropy. Let \((\nu_n)_{n \in \N}\)  be a sequence of \(\sigma\)-invariant probability measures converging on cylinders to \(\nu \in \MCM_{\leq 1}(\Sigma,\sigma)\). Then 
\[\limsup_{n \rightarrow \infty} h_{\nu_n} \leq |\nu| h_{\nu/|\nu|}+(1-|\nu| )\delta_{\infty}.\]
If the sequence converges on cylinders to the zero measure, then the right hand side is understood as \(\delta_{\infty}\).
\end{theo}

\section{Entropy relations via a suspension space}\label{sec:susspace}
For convenience, we let
\[\MCM_{\lambda < \infty}(\Sigma,\sigma):=\left\{ \nu \in \MCM(\Sigma,\sigma): \lambda_{\nu} <\infty \right\}\]
and let \(\MCM_{\MCE,\lambda < \infty}(\Sigma,\sigma)\) be the subset of \(\MCM_{\lambda < \infty}(\Sigma,\sigma)\) consisting of the ergodic measures. Recall that \(f_{\log|T'|}\) is the uniformly continuous function with \(f_{\log|T'|}=\log|T' \circ \Pi|\) everywhere except for some points in \(\Pi^{-1}(E)\). 

In this section we will construct a sequence of suspension spaces on \(\Sigma\) with locally constant roof functions \(\tau_m\) which take integer values and are such that \(\frac{\tau_m}{2^m}\) converges to \(f_{\log|T'|}\) uniformly. We will show that each suspension space is closely related to a countable Markov shift with finite topological entropy, \(\Sigma_{\tau_m}\) say, and further show that there is a correspondence between measures in  \(\MCM_{\lambda<\infty}(\Sigma,\sigma)\) and measures in \(\MCM(\Sigma_{\tau_m},\sigma)\). By Abramov's formula, for each \(\nu \in \MCM_{\lambda<\infty}(\Sigma,\sigma)\) the entropy of the corresponding measure in \(\MCM(\Sigma_{\tau_m},\sigma)\) will be approximately equal to \(h_\nu/\lambda_\nu\), with the limit converging as \(m \rightarrow \infty\). Using this and the results of CMS with finite topological entropy discussed in the previous section, we are able to prove the following two propositions:

\begin{prop}\label{prop:entropydensity}
Let \(\nu \in \MCM_{\lambda<\infty}(\Sigma,\sigma)\). Then there exists a sequence of ergodic measures \(\nu_n \in \MCM_{\lambda<\infty}(\Sigma,\sigma)\) supported on finitely many symbols such that \(\nu_n \rightarrow \nu\) weak*,
\(\lim_{n \rightarrow \infty} h_{\nu_n}=h_{\nu}\) and \(\lim_{n \rightarrow \infty} \lambda_{\nu_n}=\lambda_{\nu} \).
\end{prop}

\begin{prop}\label{prop:uppersemicontinuity}
Suppose \(\inf\{|T'(x)|:x \in \overline{I_i}\} \rightarrow \infty\) as \(i \rightarrow \infty\). Let \(\nu_n \in \MCM_{\lambda<\infty}(\Sigma,\sigma)\) be such that \(\frac{h_{\nu_n}}{\lambda_{\nu_n}}>s_{\infty}+\delta\) for some \(\delta>0\) and all \(n \in \N\). Then there exists \(\nu \in \MCM_{\lambda<\infty}(\Sigma,\sigma)\) and a subsequence \((n_k)_{k \in \N}\) such that \(\nu_{n_k} \rightarrow \nu\) weak* and 
\[\limsup_{n \rightarrow \infty} \frac{h_{\nu_n}}{\lambda_{\nu_n}} \leq \frac{h_{\nu}}{\lambda_{\nu}}.\]
\end{prop}

As well the constructed countable Markov shifts and Abramov's formula, the key elements in the proof of Proposition \ref{prop:entropydensity} are Lemma 5.1 in \cite{ITV18} and Proposition \ref{prop:entropydensitymike}. Once we have proved Proposition \ref{prop:entropydensity}, we will discuss some heuristics for the proof of Proposition \ref{prop:uppersemicontinuity}.

We now define the suspension spaces on \(\Sigma\) and prove several lemmas. For \(m \in \N\) let
\[\tau_{m}:=\sum_{(\omega_1, \ldots ,\omega_m) \in \MCC_m(\Sigma)} 1_{(\omega_1, \ldots ,\omega_m)}  k_{(\omega_1,\ldots, \omega_m)}, \]
where 
\[k_{(\omega_1,\ldots,\omega_m)}=\max \left\{k \in \N \cup \{0\}: \frac{k}{2^m} \leq \inf \{f_{\log|T'|}(\omega)|: \omega \in [\omega_1,\ldots,\omega_m]\} \right\}.\]
Then \(\frac{\tau_m}{2^m} \nearrow f_{\log|T'|}\) and each \(\tau_{m}\) is uniformly continuous, takes integer values on \(\Sigma\), and is strictly positive if \(m\) is sufficiently large. We may assume that \(\tau_m\) is strictly positive for all \(m \in \N\) (otherwise replace \(2^m\) with \(l^m\) where \(l \in \N\) is such that \(1/l<\log \zeta\)). Since \(f_{\log|T'|}\) is uniformly continuous and by (\ref{eqn:integralequality}), we have that 
\begin{equation}\label{eqn:uniformcont}
    \lim_{m \rightarrow \infty} \sup_{\nu \in \MCM(\Sigma,\sigma)} \left|\int \frac{\tau_m}{2^m} \diff \nu - \lambda_{\nu} \right|= 0.
\end{equation}

For now, fix an \(m \in \N\). We define the suspension space \(X\) by 
\[X:=\{ (\omega,x):\omega \in \Sigma, 0 \leq x \leq \tau_m(\omega)\},\]
where we identify the points \((\omega,\tau_m(\omega))=(\sigma(\omega),0)\) for all \(\omega \in \Sigma\). Let \(\varphi_{t}\) be the map given by \(\varphi_{t}(\omega,x)=(\omega,x+t)\). Then \((X, \varphi_t)\) defines a semi-flow. We will mainly consider the map \(\varphi_1\) because, due to the roof function taking integer values, \((X,\varphi_1)\) is closely related to a CMS. We endow \(X\) with the Bowen-Walters metric described in \cite[Section 2.3]{BI06}. We can define a map \( M:\MCM_{\lambda<\infty}(\Sigma,\sigma) \rightarrow \MCM(X,\varphi_{t})\)
by
\begin{equation}\label{eqn:susspacemeasuremap}
    M \nu =\frac{(\nu \times \Leb)|_X}{\int \tau_m \diff \nu},
\end{equation}
where \(\Leb\) is the Lebesgue measure on the real line \(\R\) and \((\nu \times \Leb)|_X\) stands for the restriction of \(\nu \times \Leb\) to \(X\). When we want to be explicit about the dependence of \(M\) on \(m\), we will write \(M_m\). Work by  Ambrose and Kakutani implies that \(M\) is a bijection. It is easy to see that \(M \nu \) is ergodic if and only if \(\nu\) is ergodic.

Let \(\nu \in \MCM_{\lambda< \infty}(\Sigma,\sigma)\) and let \(A \subset X\) be the set \(A=\Sigma \times [0,1)\). Then \(M \nu (A)=1/\int \tau_m \diff \nu \). Furthermore \(A\) is spanning with respect to the map \(\varphi_{1}\), so by Abramov's formula 
\begin{align}
    h_{M \nu }(X,\varphi_{1})&= M \nu (A) h_{M \nu |_A}(A,\varphi_{1}|_A) \nonumber \\
    &= \frac{h_{M \nu |_A}(A,\varphi_{1}|_A)}{\int \tau_m \diff \nu}, \label{eqn:arbamoventropy}
\end{align}
where \(M \nu |_A\) is the measure on the induced system \((A,\MCB(X)|_A, M \nu |_A, \varphi_{1}|_A)\).

\begin{lemma}\label{lem:iso1}
For \(\nu \in \MCM_{\lambda<\infty}(\Sigma,\sigma)\), \(h_{M \nu |_A}(A,\varphi_{1}|_A)=h_{\nu}(\Sigma,\sigma)\).
\end{lemma}

\begin{proof}
Notice that \((A,\MCB(X)|_A, M \nu|_A, \varphi_{1}|_A)\) is isomorphic to \((\Sigma \times [0,1), \MCB(\Sigma) \otimes \MCB([0,1)), \nu \times \text{Leb}, \sigma \times \id)\). So by Theorem 4.23 in \cite{Wal81}
\begin{align*}
    h_{M\nu|_A}(A,\varphi_{1}|_A) &=h_{\nu}(\Sigma,\sigma)+h_{\Leb}([0,1),\id) \\
    &= h_{\nu}(\Sigma,\sigma).
\end{align*}
\end{proof}

We now define an map from \(X\) onto a topologically transitive CMS, \(\Sigma_{\tau_m}\) say, with finite topological entropy. This will be related to \(\Sigma\) in the following simple way. First note that the CMS \(\Sigma\) is isomorphic to a CMS, \(\Sigma_{m}\) say, on the alphabet \(\MCC_m(\Sigma)\) where \((\omega_1,\ldots,\omega_m) \rightarrow (\omega'_1,\ldots,\omega'_m) \) if and only if \(\omega_{l+1}=\omega'_l\) for all \(1 \leq l \leq m-1\). The CMS \(\Sigma_{\tau_m}\) is then formed from \(\Sigma_m\) by replacing each vertex \((\omega_1,\ldots,\omega_m)\) by a string of \(k_{(\omega_1,\ldots,\omega_m)}\) vertices.

Let \(q\) be a bijection of \(\MCC_m(\Sigma)\) onto \(\N\). Consider the map \(\xi:X \rightarrow \N\) given by \( \xi(\omega,x)=1+ \floor{x}+\sum_{j=1}^{q(\omega_1,\ldots,\omega_m)-1} k_{q^{-1}(j)}\). Notice that for \(a \in \N\), \(\xi^{-1}(a)=[\omega_1,\ldots ,\omega_m] \times [j-1,j)\) for some \((\omega_1,\ldots,\omega_m) \in \MCC_m(\Sigma)\) and \(j \in \{1,\ldots,k_{(\omega_1,\ldots,\omega_m)}\}\). Now, define the map \(\pi: X \rightarrow \N^\N\) by \( (\pi(\omega,x))_j=\xi(\varphi_1^{j-1}(\omega,x)) \). Then \((\pi(X),\sigma)\) is the CMS with transition matrix \((t_{ab})\), where \(t_{ab}=1\) if \(a,b \in (\sum_{i=1}^{j} k_{q^{-1}(i)}, \sum_{i=1}^{j+1} k_{q^{-1}(i)}] \cap \N\) for some \(j \in \N \cup \{0\}\) and \(b=a+1\); or if \(a=\sum_{i=1}^j k_{q^{-1}(i)}\) for some \(j \in \N\), \(b=1+\sum_{i=1}^{l-1} k_{q^{-1}(i)}\) for some \(l \in \N\) and \(q^{-1}(j) \xrightarrow[\Sigma_m]{} q^{-1}(l)\); and zero otherwise. That is, \((\pi(X),\sigma)\) is the CMS \(\Sigma_{\tau_m}\) described above. For later use we also define the map \(\widetilde{\pi}: X \rightarrow \Sigma_{\tau_m} \times [0,1)\) by \(\widetilde{\pi}(\omega,x)=(\pi(\omega,x),\{x\}) \), where \(\{x\}\) is the fractional part of \(x\). It is easy to see that \(\widetilde{\pi}\) is invertible and its inverse is continuous. 

\begin{lemma}\label{lem:iso2}
For \(\nu \in \MCM(X,\varphi_{1})\), \(h_{\nu}(X,\varphi_{1})=h_{\pi^* \nu}(\Sigma_{\tau_m},\sigma)\). 
\end{lemma}
\begin{proof}
It is routine to check that \((X,\MCB(X),\varphi_{1},\nu)\) is isomorphic to \((\pi(X) \times [0,1), \MCB(\Sigma_{\tau_m})\otimes \MCB([0,1)),\pi^* \nu \times \text{Leb}, \sigma \times \id)\). Then, again by Theorem 4.23 in \cite{Wal81},
\begin{align*}
    h_{\nu}(X,T)&=h_{\pi^* \nu \times \text{Leb}}(\Sigma_{\tau_m} \times [0,1),\sigma \times \id) \\
    &=h_{\pi^*\nu}(\Sigma_{\tau_m},\sigma)+h_{\text{Leb}}([0,1),\id) \\
    &= h_{\pi^*\nu}(\Sigma_{\tau_m},\sigma).
\end{align*}
\end{proof}

Lemma \ref{lem:iso1} and Lemma \ref{lem:iso2} together with equation (\ref{eqn:arbamoventropy}) gives 
\begin{equation}\label{eqn:susspaceentropyrelation}
    h_{\pi^* M \nu}(\Sigma_{\tau_m},\sigma)=\frac{h_{\nu}(\Sigma,\sigma)}{\int \tau_m \diff \nu}
\end{equation}
for all \(\nu \in \MCM_{\lambda< \infty}(\Sigma,\sigma)\). We will use (\ref{eqn:susspaceentropyrelation}) numerous times in what follows.

\begin{lemma}\label{lem:measuremapsurjective}
The map \(\pi^* \circ M : \MCM_{\MCE,\lambda<\infty}(\Sigma,\sigma) \rightarrow \MCM_{\MCE}(\Sigma_{\tau_m},\sigma)\) is surjective. 
\end{lemma}

\begin{proof}
Let \(\eta \in \MCM_{\MCE}(\Sigma_{\tau_m},\sigma)\). The measure \((\widetilde{\pi}^{-1})^* (\eta \times \Leb_{[0,1)}) \in \MCM(X,\varphi_t)\) is ergodic with respect to the flow \(\varphi_t\), hence the measure \(\nu:=M^{-1} (\widetilde{\pi}^{-1})^* (\eta \times \Leb_{[0,1)}) \in \MCM_{\lambda<\infty}(\Sigma,\sigma)\) is also ergodic. It is clear that \(\pi^* M \nu=\eta\).
\end{proof}

\begin{lemma}\label{lem:finitetopentropy}
    The topological entropy \(h_\text{top}(\Sigma_{\tau_m},\sigma)\) is finite. 
\end{lemma}

\begin{proof}
This follows by (\ref{eqn:susspaceentropyrelation}), Lemma \ref{lem:measuremapsurjective}, the variational principle (\ref{eqn:hvariationalprinciple}), and since \(h_{\mu} \leq \lambda_{\mu}\) for all \(\MCM_{\MCE}(\Sigma,\sigma)\). 
\end{proof}

We are now ready to prove Proposition \ref{prop:entropydensity}. 

\begin{proof}[Proof of Proposition \ref{prop:entropydensity}]
Let \(\nu \in  \MCM_{\lambda<\infty}(\Sigma,\sigma)\). By Proposition \ref{prop:entropydensitymike}, for each \(m \in \N\) there exists a sequence of ergodic measures \(\eta_{m,n} \in \MCM_{\MCE}(\Sigma_{\tau_m},\sigma)\) supported on finitely many symbols such that \(\eta_{m,n} \rightarrow \pi^* M_m \nu\) weak* and 
\begin{equation*}\label{eqn:sequenceofmeasures}
    \lim_{n \rightarrow \infty} h_{\eta_{m,n}}(\Sigma_{\tau_m},\sigma)=  h_{\pi^*M_m \nu}(\Sigma_{\tau_m},\sigma)=\frac{h_{\nu}(\Sigma,\sigma)}{\int \tau_m \diff \nu}.
\end{equation*}
The measures \(\nu_{m,n}:=M_m^{-1} (\widetilde{\pi}^{-1})^* (\eta_{m,n} \times \Leb_{[0,1)}) \in \MCM(\Sigma,\sigma)\) are ergodic and supported on finitely many symbols. Moreover,
since \(\widetilde{\pi}^{-1}\) is continuous and by Lemma 5.1 in \cite{ITV18}, the sequences \((\nu_{m,n})_{n \in \N}\) converge weak* to \(\nu\) and satisfy
\[\lim_{n \rightarrow \infty} \int \tau_m \diff \nu_{m,n} = \int \tau_m \diff \nu.\]
Equation (\ref{eqn:susspaceentropyrelation}) further gives that
\begin{align*}
    \lim_{n \rightarrow \infty} h_{\nu_{m,n}}(\Sigma,\sigma) &= \lim_{n \rightarrow \infty} h_{\eta_{m,n}}(\Sigma_{\tau_m,\sigma}) \int \tau_m \diff \nu_{m,n}  \\ 
    &= \frac{h_{\nu}(\Sigma,\sigma)}{\int \tau_m \diff \nu} \int \tau_m \diff \nu \\
    &= h_{\nu(\Sigma,\sigma)}.
\end{align*}
We now apply the Monotone Convergence Theorem and a diagonal argument to get the desired sequence. It is well known that, since \(\Sigma\) is separable, the weak* topology on \(\MCM(\Sigma,\sigma)\) is metrisable. Let \(d:\MCM(\Sigma,\sigma)^2 \rightarrow [0,\infty)\) be a given metric (take \(d\) to be the Lévy–Prokhorov metric, for example). For each \(m\), let \(n_m\) be such that \(\left|\int \frac{\tau_m}{2^m} \diff \nu_{m,n_m}- \int \frac{\tau_m}{2^m} \diff \nu \right| < \frac{1}{m}\), \(|h_{\nu_{m,n_m}}(\Sigma,\sigma)-h_{\nu}(\Sigma,\sigma)|<\frac{1}{m} \), and \(d(\nu_{m,n_m},\nu)<\frac{1}{m}\). Then clearly \(\nu_{m,n_m} \rightarrow \nu\) weak* and \(h_{\nu_{m,n_m}} \rightarrow h_{\nu}\). Moreover,
\begin{align*}
    \left|\lambda_{\nu_{m,n_m}}-\lambda_{\nu} \right| \leq \left|\lambda_{\nu_{m,n_m}}-\int \frac{\tau_m}{2^m} \diff \nu_{m,n_m}\right|+
    \left|\int \frac{\tau_m}{2^m} \diff \nu_{m,n_m}-\int \frac{\tau_m}{2^m} \diff \nu \right|+ \left|\int \frac{\tau_m}{2^m} \diff\nu- \lambda_{\nu} \right|,
\end{align*}
which, in view of equation (\ref{eqn:uniformcont}), converges to 0 as \(m \rightarrow \infty\).
\end{proof}

For the rest of this section, we assume that 
\begin{equation}\label{eqn:assumptioninsusspace}
    \inf\{|T'(x)|:x \in \overline{I_i} \setminus E\} \rightarrow \infty
\end{equation}
as \(i \rightarrow \infty\). To prove Proposition \ref{prop:uppersemicontinuity} we require the following lemma which follows from routine arguments involving the pressure. Since the proof is relatively long and unrelated to the countable Markov shifts \(\Sigma_{\tau_m}\), we will postpone the proof of this lemma until Section \ref{sec:uiexact} (see Lemmas \ref{lem:approxsinfinity} and \ref{lem:boundedlyapunov}). 
\begin{lemma}\label{lem:sinfinitycharacterisation}
\(s_{\infty}=\sup_{(\nu_n)_{n \in \N} \subset \MCM(\Sigma,\sigma)} \left\{\limsup_{n \rightarrow \infty} \frac{h_{\nu_n}}{\lambda_{\nu_n}}:\lambda_{\nu_n} \rightarrow \infty \right\}\).
\end{lemma}

Before we continue, let us give some heuristics for the proof of Proposition \ref{prop:uppersemicontinuity}. Using Lemma \ref{lem:sinfinitycharacterisation} it can be shown that, since the measures \(\nu_n\) in Proposition 5.2 satisfy \(h_{\nu_n}/\lambda_{\nu_n}>s_{\infty}+\delta\) for all \(n \in \N\),
they must be tight and hence have a weak* limit point, \(\nu\) say. Assume that \(\nu_n \rightarrow \nu\) weak* (note this can be done without loss of generality). One may hope to use Proposition 8.5 in \cite{ITV19} to prove Proposition \ref{prop:uppersemicontinuity}. Unfortunately, the measures \(\nu_n\) converging weak* does not imply that the measures \(M \nu_n\) converge weak*, nor even is the set of measures \((M \nu_n)_{n \in \N}\) necessarily tight. To see this, consider the measures \(\rho_n:=(1-\varepsilon_n)\nu_n+\varepsilon_n \eta_n\), where \(\eta_n \in \MCM(\Sigma,\sigma)\) is a sequence of measures converging on cylinders to the zero measure and \(\varepsilon_n\) is a positive sequence converging to zero. Clearly \(\rho_n\) also converges weak* to \(\nu\), but as the map \(M\) gives a disproportionate amount of weight to cylinders where \(\log |T' \circ \Pi|\) is large, the part of the measure corresponding to \(\varepsilon_n \eta_n\) on the suspension space may not decay. However, as \(h_{\nu_n}/\lambda_{\nu_n}>s_{\infty}+\delta\) and by Lemma \ref{lem:sinfinitycharacterisation},
\begin{align*}
    \limsup_{n \rightarrow \infty} \frac{h_{\rho_n}}{\lambda_{\rho_n}}&= \limsup_{n \rightarrow \infty}  \frac{(1-\varepsilon_n)h_{\nu_n}+\varepsilon_n h_{\eta_n}}{(1-\varepsilon_n)\lambda_{\nu_n}+\varepsilon_n \lambda_{\eta_n}} \\ &\leq \limsup_{n \rightarrow \infty} \frac{h_{\nu_n}}{\lambda_{\nu_n}}.
\end{align*}
That is, intuitively, `small' bits of measure going off to infinity should not hinder upper semi-continuity from holding. 

While we cannot split the measures \(\nu_n\) up into parts that are staying bounded and parts going off to infinity, the behaviour of `small' bits of measure going off to infinity on \(\Sigma\) is captured by loss of mass in the topology of convergence on cylinders on \(\Sigma_{\tau_m}\). In particular, for each \(m \in \N\), by compactness \(\pi^* M_m \nu_n\) has a limit point in the topology of convergence on cylinders, \(\eta_m \in \MCM_{\leq 1}(\Sigma_{\tau_m})\) say. We will show that the measures \(\nu^{(m)}:=M_m^{-1} (\widetilde{\pi}^{-1})^* \left( \frac{\eta_m}{|\eta_m|} \times \Leb_{[0,1)} \right)\) are \textit{precisely} equal to \(\nu\).  Moreover, by Lemma \ref{lem:sinfinitycharacterisation} and (\ref{eqn:susspaceentropyrelation}) one can see that \(s_{\infty}\) is closely related to the quantities \(\delta_{\infty}(\Sigma_{\tau_m},\sigma)\). Hence using that \(\pi^* M_m \nu_n \rightarrow \eta_m\) on cylinders, Theorem \ref{theo:uppersemicontiuitymike}, and (\ref{eqn:susspaceentropyrelation}), we are able to prove the upper semi-continuity statement. 

The next two lemmas show that we can relate cylinders in \(\Sigma\) to cylinders in \(\Sigma_{\tau_m}\), and vice versa. In Lemma \ref{lem:sinfinitydeltainfinity} we will use this to relate \(s_{\infty}\) with the quantities \(\delta_{\infty}(\Sigma_{\tau_m},\sigma)\). Later, Lemma 5.9 will be used again to prove that the measures \(\nu^{(m)}\) are indeed equal to the weak* limit point \(\nu\). We remark that we do this directly by showing that, along some subsequence \(n_k\), \(\nu_{n_k} \rightarrow \nu^{(m)}\) on cylinders for every \(m\). Thus, it is not necessary for us to use the tightness of the measures \(\nu_n\) to deduce the existence of the weak* limit point \(\nu\).

\begin{lemma}\label{lem:cylindersgobacktocylinders}
For every cylinder \([\omega_1,\ldots,\omega_k] \subset \Sigma_{\tau_m}\) there exists a cylinder \([\omega'_1,\ldots,\omega'_{k'}] \subset \Sigma\) such that for all \(\nu \in \MCM(\Sigma,\sigma)\)
\[\pi^* M_m \nu([\omega_1,\ldots,\omega_k])= \frac{\nu([\omega'_1,\ldots,\omega'_{k'}])}{\int \tau_m \diff \nu}.\]
\end{lemma}
\begin{proof}
 Let \([\omega_1,\ldots,\omega_k] \subset \Sigma_{\tau_m}\). Note that we may assume that \(\omega_1=1+\sum_{j=1}^{l-1} k_{q^{-1}(j)}\) for some \(l \in \N\), as if \(1+\sum_{j=1}^{l-1} k_{q^{-1}(j)} < \omega_1 \leq \sum_{j=1}^{l} k_{q^{-1}(j)}\) for some \(l\) then for any \(\eta \in \MCM(\Sigma_{\tau_m},\sigma)\)
 \begin{align*}
     \eta([\omega_1,\ldots,\omega_k])&=\eta \left(\sigma^{-(\omega_1-1-\sum_{j=1}^{l-1} k_{q^{-1}(j)})} \left[\omega_1,\ldots,\omega_k \right] \right) \\
     &= \eta \left(\left[ 1+\sum_{j=1}^{l-1} k_{q^{-1}(j)},2+\sum_{j=1}^{l-1} k_{q^{-1}(j)},\ldots,\omega_1,\ldots,\omega_k \right] \right).
 \end{align*}
Similarly we may assume that \(\omega_k=\sum_{j=1}^{l'} k_{q^{-1}(j)}\) for some \(l' \in \N\), since if \(1+\sum_{j=1}^{l'-1} k_{q^{-1}(j)} \leq \omega_k < \sum_{j=1}^{l'} k_{q^{-1}(j)}\) for some \(l' \in \N\), then 
\[[\omega_1,\ldots,\omega_k]=\left[\omega_1,\ldots,\omega_k,\omega_k+1,\ldots, \sum_{j=1}^{l'} k_{q^{-1}(j)}\right].\]
We can therefore write \([\omega_1,\ldots,\omega_k]\) more explicitly as \[\left[1+\sum_{j=1}^{l_1-1}k_{q^{-1}(j)},\ldots,\sum_{j=1}^{l_1}k_{q^{-1}(j)},1+\sum_{j=1}^{l_2-1}k_{q^{-1}(j)},\ldots,\sum_{j=1}^{l_2}k_{q^{-1}(j)}, \ldots, 1+\sum_{j=1}^{l_{k'}-1}k_{q^{-1}(j)},\ldots,\sum_{j=1}^{l_{k'}}k_{q^{-1}(j)}  \right]\]
for some \((l_1,\ldots,l_{k'}) \in \N^{k'}\) where \(k' \leq k\). Then \(\pi^{-1}([\omega_1,\ldots,\omega_k])\) is equal to \([\omega'_1,\ldots,\omega'_{k'+m-1}] \times [0,1) \subset X \) for some cylinder \([\omega'_1,\ldots,\omega'_{k'+m-1}] \subset \Sigma\). Hence
\begin{align*}
    \pi^* M_m \nu([\omega_1,\ldots,\omega_k] )&=M_m \nu([\omega'_1,\ldots,\omega'_{k'+m-1}] \times [0,1) ) \\
    &=\frac{\nu([\omega'_1,\ldots,\omega'_{k'+m-1}])}{\int \tau_m \diff \nu}.
\end{align*}
\end{proof}

\begin{lemma}\label{lem:cylindersgotocylinders}
For every cylinder \([\omega_1,\ldots,\omega_k] \subset \Sigma\) with \(k \geq m\), there exists \([\omega'_1,\ldots, \omega'_{k'}] \subset \Sigma_{\tau_m} \) such that for all \(\nu \in \MCM(\Sigma,\sigma)\)
\[\nu([\omega_1,\ldots,\omega_k])=\pi^*  M_m \nu ([\omega'_1,\ldots,\omega'_{k'}]) \int \tau_m \diff \nu.\] 
\end{lemma}

\begin{proof}
By (\ref{eqn:susspacemeasuremap}) we have 
\[M_m \nu ([\omega_1,\ldots,\omega_k] \times [0,1))=\frac{\nu([\omega_1,\ldots,\omega_k])}{\int \tau_m \diff \nu}.\]
It is easy to see that \(\pi([\omega_1,\ldots,\omega_k] \times [0,1))=[\omega'_1,\ldots,\omega'_{k'}]\) for some cylinder \([\omega'_1,\ldots,\omega'_{k'}] \subset \Sigma_{\tau_m}\), where \(k' \geq k-m+1\). Thus 
\[\nu([\omega_1,\ldots,\omega_k])=\pi^* M_m \nu ([\omega'_1,\ldots,\omega'_{k'}]) {\int \tau_m \diff \nu}.\]
\end{proof}

\begin{lemma}\label{lem:sinfinitydeltainfinity}
\(\lim_{m \rightarrow \infty} 2^m \delta_{\infty}(\Sigma_{\tau_m},\sigma)= s_{\infty}\).
\end{lemma}
\begin{proof}
We first show \(s_{\infty} \leq 2^m \delta_{\infty}(\Sigma_{\tau_m},\sigma)\) for any \(m \in \N\). Fix \(m \in \N\). By Lemma \ref{lem:approxsinfinity} we can find a sequence \(\nu_n \in \MCM(\Sigma,\sigma)\) such that \(\frac{h_{\nu_n}}{\lambda_{\nu_n}} \rightarrow s_{\infty} \) and \(\lambda_{\nu_n} \rightarrow \infty\). Then by Lemma \ref{lem:cylindersgobacktocylinders}, for any cylinder \([\omega_1,\ldots,\omega_k] \subset \Sigma_{\tau_m}\)
\begin{align*}
    \pi^* M \nu_n([\omega_1,\ldots, \omega_k]) &=\frac{\nu_n([\omega'_1,\ldots,\omega'_{k'}])}{\int \tau_m \diff \nu_n}  \\
    &\leq \frac{1}{\int \tau_m \diff \nu_n} \xrightarrow[n \rightarrow \infty]{} 0 
\end{align*}
and so \(\pi^* M \nu_n \rightarrow 0\) on cylinders. Furthermore 
\begin{align*}
    \delta_{\infty}(\Sigma_{\tau_m},\sigma) &\geq \limsup_{n \rightarrow \infty} h_{\pi^* M \nu_n }(\Sigma_{\tau_m},\sigma) \\
    &= \limsup_{n \rightarrow \infty} \frac{h_{\nu_n}(\Sigma,\sigma)}{\int \tau_m \diff \nu_n} \\
    &\geq \limsup_{n \rightarrow \infty} \frac{h_{\nu_n}(\Sigma,\sigma)}{2^m\lambda_{\nu_n}} \\
    &= \frac{s_{\infty}}{2^m}.
\end{align*}

Let \(\varepsilon>0\). We now show that for all \(m\) large enough \(2^m \delta_{\infty}(\Sigma_{\tau_m},\sigma) \leq s_{\infty}+\varepsilon\). Fix \(m \in \N\), which we will assume is large so that the supremum in (\ref{eqn:uniformcont}) is small, and suppose for a contradiction that \(2^m \delta_{\infty}(\Sigma_{\tau_m},\sigma)>s_{\infty}+\varepsilon\). We may find \(\eta_n \in \MCM(\Sigma_{\tau_m},\sigma)\) such that \(h_{\eta_n}(\Sigma_{\tau_m},\sigma) \rightarrow \delta_{\infty}(\Sigma_{\tau_m},\sigma)\) and \(\eta_n \rightarrow 0\) on cylinders. Then for \(\nu_n:=M^{-1} (\widetilde{\pi}^{-1})^*(\eta_n \times \Leb) \in \MCM(\Sigma,\sigma)\) and all \(n\) sufficiently large we have
\[2^m \frac{h_{\nu_n}(\Sigma,\sigma)}{\int \tau_m \diff \nu_n} > s_{\infty}+\varepsilon.\]
Assuming we had chosen our \(m\) large enough, for these \(n\) we can further have 
\[\frac{h_{\nu_n}(\Sigma,\sigma)}{\lambda_{\nu_n}} > s_{\infty}+\frac{\varepsilon}{2}.\]
Thus Lemma \ref{lem:sinfinitycharacterisation} implies that the sequence \((\lambda_{\nu_n})_{n \in \N}\) is bounded, which in turn implies that the sequence \((\int \tau_m \diff \nu_n)_{n \in \N}\) is bounded. However, by Lemma \ref{lem:cylindersgotocylinders}, for any \(m\)th level cylinder \([\omega_1,\ldots,\omega_m] \subset \Sigma\) \[\frac{\nu_n([\omega_1,\ldots,\omega_m])}{ \int \tau_m \diff \nu_n }=\eta_n ([\omega'_1,\ldots,\omega'_{k'}])\]
for some cylinder \([\omega'_1,\ldots,\omega'_{k'}] \subset \Sigma_{\tau_m}\). As \(\eta_n \rightarrow 0\) on cylinders and \((\int \tau_m \diff \nu_n)_{n \in \N}\) is bounded, this implies that we must have 
\[\nu_n([\omega_1,\ldots,\omega_m]) \xrightarrow[n \rightarrow \infty]{} 0.\]
Since this holds for all \(m\)th level cylinders \([\omega_1,\ldots,\omega_m] \subset \Sigma\), by our assumption (\ref{eqn:assumptioninsusspace}) it follows that \(\lambda_{\nu_n} \rightarrow \infty\). This is the desired contradiction.
\end{proof}

\begin{lemma}\label{lem:mlimitoflyapunov}
Let \(\eta_n \in \MCM(\Sigma_{\tau_m},\sigma)\) and \(\eta \in \MCM_{\leq 1}(\Sigma_{\tau_m},\sigma) \setminus\{0\}\) be such that \(\eta_n \rightarrow \eta\) on cylinders. Let \(\nu_n:=M^{-1}(\widetilde{\pi}^{-1})^* (\eta_n \times \Leb)\) and \(\nu:=M^{-1}(\widetilde{\pi}^{-1})^* (\frac{\eta}{|\eta|} \times \Leb) \). Then
\[\int \tau_m \diff \nu_n \xrightarrow[n \rightarrow \infty]{} \frac{1}{|\eta|} \int \tau_m \diff \nu.\]
\end{lemma}

Note that when \(|\eta|=1\) we recover that \(\int \tau_m \diff \nu_n \rightarrow \int \tau_m \diff \nu\), as was argued in the proof of Proposition \ref{prop:entropydensity} (that is, as one would expect from Lemma 5.1 in \cite{ITV18}). This lemma extends this to the case where mass is lost.

\begin{proof}
Recalling how we defined the map \(q: \MCC_m(\Sigma) \rightarrow \N\), we can write each \(a \in \N\) uniquely as 
\(a=j+\sum_{i=1}^{q(\omega_1,\ldots,\omega_m)-1} k_{q^{-1}(i)}\)
for some \((\omega_1,\ldots,\omega_m) \in \MCC_m(\Sigma)\) and \(j \in \{1, \ldots, k(\omega_1,\ldots,\omega_m)\}\). Let \(q':\N \rightarrow \MCC_m(\Sigma)\) be the map that takes \(a \in \N\) to this \((\omega_1,\ldots,\omega_m)\) and let \(f: \Sigma_{\tau_m} \rightarrow \R\) be the function given by \(f(\omega)=\frac{1}{k_{q'(\omega_1)}}\). Then \(f\) is a locally constant function which depends only on the first digit and \(\lim_{n \rightarrow \infty} \sup_{\omega \in [n]} |f(\omega)|=0\). Thus \(f \in C_0(\Sigma_{\tau_m})\), so we have 
\[\int f \diff \eta_n \xrightarrow[n \rightarrow \infty]{} \int f \diff \eta=|\eta| \int f \diff {\left(\frac{\eta}{|\eta|} \right)}.\]
Define the function \(\widetilde{f}:\Sigma_{\tau_m} \times [0,1) \rightarrow \R\) by \(\widetilde{f}(\omega,t)=f(\omega)\). We have 
\[\int \widetilde{f} \diff{(\eta_n \times \Leb)} \xrightarrow[n \rightarrow \infty]{} |\eta| \int \widetilde{f} \diff \!\left(\frac{\eta}{|\eta|} \times \Leb \right),\]
and so 
\begin{equation}\label{eqn:convergenceonsusspace}
    \int \widetilde{f}\circ \widetilde{\pi} \diff((\widetilde{\pi}^{-1})^* (\eta_n \times \Leb)) \xrightarrow[n \rightarrow \infty]{} |\eta| \int \widetilde{f} \circ \widetilde{\pi} \diff{\left((\widetilde{\pi}^{-1})^* \left(\frac{\eta}{|\eta|} \times \Leb \right) \right)}.
\end{equation}
By Kac's lemma, for each \(\rho \in \MCM(X, \varphi_t)\),
\[ \int_{X}\widetilde{f}\circ \widetilde{\pi} \diff\rho = \frac{\int_{\Sigma} \left( \int_0^{\tau_m(\omega)} \widetilde{f}\circ \widetilde{\pi}(\omega,t) \diff t \right) \diff(M^{-1}\rho) }{\int_{\Sigma} \tau_m \diff (M^{-1}\rho)}. \]
Note that
\[\int_0^{\tau_m(\omega)} \widetilde{f}\circ \widetilde{\pi}(\omega,t) \: \diff t =\int_0^{k_{(\omega_1,\ldots , \omega_m)}} \frac{1}{k_{(\omega_1,\ldots , \omega_m)}} \: \diff t = 1\]
for all \(\omega \in \Sigma\). Thus, for any \(\rho \in \MCM(X, \varphi_t)\) 
\[\int_{X}\widetilde{f}\circ \widetilde{\pi} \diff \rho = \frac{1}{\int_{\Sigma} \tau_m \diff(M^{-1}\rho)}. \] 
Finally, with \(\nu_n:=M^{-1}(\widetilde{\pi}^{-1})^* (\eta_n \times \Leb)\) and \(\nu:=M^{-1}(\widetilde{\pi}^{-1})^* (\frac{\eta}{|\eta|} \times \Leb) \), this together with equation (\ref{eqn:convergenceonsusspace}) implies 
\[\int \tau_m \: \diff \nu_n \xrightarrow[n \rightarrow \infty]{} \frac{1}{|\eta|} \int \tau_m \diff \nu.\]
\end{proof}

\begin{proof}[Proof of Proposition \ref{prop:uppersemicontinuity}]
Let \(\nu_n \in \MCM_{\lambda< \infty}(\Sigma,\sigma)\) be such that \(\frac{h_{\nu_n}}{\lambda_{\nu_n}}>s_{\infty}+\delta\) for all \(n \in \N\). Restricting to a subsequence if necessary, we may assume that \(\lim_{n \rightarrow \infty} \frac{h_{\nu_n}}{\lambda_{\nu_n}}=\limsup_{n \rightarrow \infty} \frac{h_{\nu_n}}{\lambda_{\nu_n}}\). For \(m=1\), since \(\Sigma_{\tau_1}\) has finite topological entropy and by Theorem 1.2 in \cite{IV19}, there exists \(\eta_1 \in \MCM_{\leq 1}(\Sigma_{\tau_1},\sigma)\) and a subsequence \((n_{1,l})_{l \in \N}\) such that \(\pi^* M_1 \nu_{n_{1,l}} \rightarrow \eta_1 \) on cylinders. Similarly, for \(m>1\) there exists \(\eta_m \in \MCM_{\leq 1}(\Sigma_{\tau_m},\sigma)\) and a subsequence \((n_{m,l})_{l \in \N}\subseteq (n_{1,l})_{l \in \N} \) such that \(\pi^* M_m \nu_{n_{m,l}} \rightarrow \eta_m \) on cylinders. By Theorem \ref{theo:uppersemicontiuitymike} and (\ref{eqn:susspaceentropyrelation}), 
\begin{equation}\label{eqn:uppersemicontapplied}
    \frac{s_{\infty}+\delta}{2^m} \leq \liminf_{n \rightarrow \infty} \frac{h_{\nu_n}}{\int \tau_m \diff \nu_n} \leq \limsup_{l \rightarrow \infty} \pi^* M_m \nu_{n_{m,l}}  \leq |\eta_m| h_{\frac{\eta_m}{|\eta_m|}}+(1-|\eta_m|) \delta_{\infty}(\Sigma_{\tau_m},\sigma).
\end{equation}
Since \(\frac{h_{\nu_n}}{\lambda_{\nu_n}}>s_{\infty}+\delta\) for all \(n \in \N\), by Lemma \ref{lem:boundedlyapunov} the sequence \((\lambda_{\nu_n})_{n \in \N}\) must be bounded. Using the argument in the proof of Lemma \ref{lem:sinfinitydeltainfinity}, this implies that \(\eta_m\) cannot be the zero measure for any \(m \in \N\) (otherwise \(\nu_n([\omega_1,\ldots, \omega_m]) \rightarrow 0\) for all \([\omega_1,\ldots, \omega_m] \in \MCC_m(\Sigma)\), implying that \(\lambda_{\nu_n}\) is unbounded).
Thus, for each \(m \in \N\) we can define \(\nu^{(m)}:=M_m^{-1} (\widetilde{\pi}^{-1})^* \left( \frac{\eta_m}{|\eta_m|} \times \Leb \right)\). Given any cylinder \([\omega_1,\ldots,\omega_k] \subset \Sigma \), where \(k \geq m\), by Lemma \ref{lem:cylindersgotocylinders} and Lemma \ref{lem:mlimitoflyapunov}
\begin{align*}
    \lim_{l \rightarrow \infty} \nu_{n_{m,l}}([\omega_1,\ldots,\omega_k])  &=\lim_{l \rightarrow \infty} \pi^* M_m \nu_{n_{m,l}} ([\omega'_1,\ldots, \omega'_{k'}])  \int \tau_m \diff \nu_n \\
    &= \frac{1}{|\eta_m|} \eta_m([\omega'_1,\ldots, \omega'_{k'}])  \int \tau_m \diff \nu^{(m)}  \\
    &= \nu^{(m)}([\omega_1,\ldots,\omega_k]),
\end{align*}
where \([\omega'_1,\ldots, \omega'_{k'}]\) is some cylinder in \(\Sigma_{\tau_m} \). For \(m=1\), this shows that \(\nu_{n_{1,l}} \rightarrow \nu^{(1)}\) on cylinders. Since \((n_{m,l})_{l \in \N} \subseteq (n_{1,l})_{l \in \N}\), clearly \(\nu_{n_{m,l}} \rightarrow \nu^{(1)}\) on cylinders. As \(\nu_{n_{m,l}}([\omega_1,\ldots,\omega_k]) \rightarrow \nu^{(m)}([\omega_1,\ldots,\omega_k]) \) for any \([\omega_1,\ldots,\omega_k] \in \cup_{k \geq m} \MCC_k(\Sigma)\), this implies that \(\nu^{(m)}=\nu^{(1)}\). We may therefore drop the dependence of \(\nu^{(m)}\) on \(m\) and write \(\nu:=\nu^{(m)}\). By Lemma 3.17 in \cite{IV19} we further have that \(\nu_{n_{1,l}} \rightarrow \nu\) weak*.

By Lemma \ref{lem:sinfinitydeltainfinity} and (\ref{eqn:uppersemicontapplied}), we have 
\[s_{\infty}+\delta \leq 2^m \liminf_{n \rightarrow \infty} \frac{h_{\nu_n}}{\int \tau_m \diff \nu_n} \leq 2^m |\eta_m| \frac{h_{\nu}}{\int \tau_m \diff \nu}+(1-|\eta_m|) \left(s_{\infty}+\varepsilon(m) \right)\]
with \(\varepsilon(m) \rightarrow 0\) as \(m \rightarrow \infty\). This implies that for all \(m\) large enough
\[2^m \frac{h_{\nu}}{\int \tau_m \diff \nu}\geq s_{\infty}+\delta,\]
and consequently
\[2^m \liminf_{n \rightarrow \infty} \frac{h_{\nu_n}}{\int \tau_m \diff \nu_n} \leq 2^m \frac{h_{\nu}}{\int \tau_m \diff \nu}.\]
Hence by the Monotone Convergence Theorem,
\begin{align*}
    \frac{h_{\nu}}{\lambda_{\nu}} &= \inf_{m \in \N} \frac{h_{\nu}}{2^{-m} \int \tau_m \diff \nu} \\
    &\geq \inf_{m \in \N} \liminf_{n \rightarrow \infty} \frac{h_{\nu_n}}{2^{-m}\int \tau_m \diff \nu_n} \\
    &\geq \liminf_{n \rightarrow \infty} \inf_{m \in \N}  \frac{h_{\nu_n}}{2^{-m}\int \tau_m \diff \nu_n} \\
    &=\lim_{n \rightarrow \infty} \frac{h_{\nu_n}}{\lambda_{\nu_n}}.
\end{align*}
This completes the proof of Proposition \ref{prop:uppersemicontinuity}.
\end{proof}

\section{Proof of Theorem \ref{theo:uiapprox}}
The following proposition proves the first statement in Theorem \ref{theo:uiapprox}.

\begin{prop}\label{prop:notinZimpliesempty}
If \(\ugamma \not\in Z\), then \(\Lambda_{\MCR}(\underline{\gamma})=\emptyset.\)
\end{prop}

\begin{proof}
We adapt the proof of Proposition 4.1 in \cite{FJLR15}. Given \(\ugamma\) assume there exists \(x \in \Lambda_{\MCR}(\ugamma)\) such that \(\lim_{n \rightarrow \infty} A_n \phi_i(x)=\gamma_i\) for all \(i \in \N\). Let \(\omega \in \Sigma\) satisfy \(\Pi \omega =x\). Let \(a \in \N\) be such that \(\omega_i=a\) infinitely often. We may assume \(\omega_1=a\) since for any \(j \in \N\) we have \(\lim_{n \rightarrow \infty} A_n \phi_i(T^j(x))=\gamma_i\). If we fix \(\varepsilon>0\) and \(k \in \N\), then by Lemma \ref{lem:variationsgoingtozero} we can find \(N \in \N\) such that for all \(n \geq N\)
\[\sup_{1 \leq i \leq k} |A_n \phi_i(x)-\gamma_i|\leq \varepsilon/2 \]
\begin{equation*}
    \sup_{1 \leq i \leq k} \sup_{y \in C_n(\omega)} |A_n \phi_i(x)-A_n \phi_i(y)| \leq \varepsilon/2.
\end{equation*}
For some \(n > N\) we must have \(\omega_n=\omega_1\). Let \(\nu\) be the shift invariant probability measure on \(\Sigma\) defined on the periodic orbit \(\overline{(\omega_1,\ldots,\omega_{n-1})}\). The measure \(\mu=\nu \circ \Pi^{-1}\) satisfies \(|\int \phi_i \diff \mu - \gamma_i| \leq \varepsilon\) for each \(1 \leq i \leq k\). This finishes the proof.
\end{proof}

\subsection{Upper Bound}
We now prove the upper estimate.
\begin{prop}\label{prop:approxupperbound}
For \(\gamma \in Z\) we have
\[\dim \Lambda_{\MCR}(\underline{\gamma}) \leq \alpha_2(\ugamma).\]
\end{prop}

We adapt the proof from the proof of the upper bound of Theorem 1.1 in \cite{FJLR15}. Note that as we are considering subsets of \(\Lambda_{\MCR}\), we are still able to define the \(\sigma^n\)-invariant Bernoulli measures used in the proof. 

Let 
\[ \MCP_{n,a}:=\{ (\omega_1,\ldots,\omega_n) \in \MCC_n(\Sigma) : \omega_1=a, (\omega_n,a) \text{ is admissible}\}\]
Fix \(k \in \N\) and \(\varepsilon>0\). For \(\ugamma=(\gamma_1,\ldots,\gamma_k) \in \R^k\) and \(n \in \N\), denote by \(\MCP_{n,a}(\ugamma)\) the set of intervals 
\[\{(\omega_1,\ldots,\omega_n) \in \MCP_{n,a}:A_n \phi_i(x) \in (\gamma_i-\varepsilon,\gamma_i+\varepsilon), \forall x \in C_n(\omega_1,\ldots,\omega_n), \forall 1 \leq i \leq k \}.\]

\begin{proof}[Proof of Proposition \ref{prop:approxupperbound}]
We can write 
\begin{equation*}
    \Lambda_{\MCR}(\ugamma)=\bigcup_{a=1}^{\infty} \bigcup_{j=0}^{\infty} T^{-j} ( \Lambda_{\MCR}(\ugamma) \cap \{\Pi \omega \in \Lambda: \omega_1=a, \omega_i=a \text{ infinitely often} \}).
\end{equation*}
Hence, as the \(T_j\) are bi-Lipschitz on \(\overline{I_j}\) and by the countable stability of Hausdorff dimension, it suffices to prove an upper bound for the sets
\begin{equation*}
    \Lambda_{\MCR,a}(\ugamma):=\Lambda_{\MCR}(\ugamma) \cap \{\Pi \omega \in \Lambda: \omega_1=a, \omega_i=a \text{ infinitely often} \}.
\end{equation*}
Fix \(a \in \N\) and let \(\tilde{s}=\dim \Lambda_{\MCR,a}(\ugamma)\). Given \(\varepsilon>0\) and \(k \in \N\), it follows from Lemma \ref{lem:variationsgoingtozero} that the  basic intervals corresponding to the set \(\bigcup_{n \geq l} \MCP_{n,a}(\ugamma)\) is a covering for \(\Lambda_{\MCR,a}(\ugamma)\) for any \(l \in \N\). We must have 
\[\sum_{I \in \MCP_{n,a}(\ugamma)} \diam(\Pi(I))^{\tilde{s}- \varepsilon} > 1 \]
for infinitely many \(n\). This must be true as otherwise for any \(\delta>0\) we would have
\begin{align*}
    \MCH^{\tilde{s}-\varepsilon+\delta}(\Lambda_{\MCR,a}(\ugamma)) &\leq \lim_{l \rightarrow \infty} \sum_{n \geq l} \sum_{I \in \MCP_{n,a}(\ugamma)} \diam(\Pi(I))^{\tilde{s}- \varepsilon+\delta} \\
    &\leq \lim_{l \rightarrow \infty} \sum_{n \geq l} \frac{1}{\zeta^{n \delta}} <\infty,
        \end{align*}
contradicting that \(\tilde{s}=\dim \Lambda_{\MCR,a}(\ugamma)\). For these \(n\) we can choose a finite subfamily \(\MCS_n \subset \MCP_{n,a}(\ugamma)\) such that the sum of their diameters in power of \(\tilde{s}-\varepsilon\) is still greater than 1. We can then choose a different exponent \(s_n>\tilde{s}-\varepsilon\) for which this sum is equal to 1. Thus, for these \(n\) we can define a \(\sigma^n\)-invariant Bernoulli measure \(\eta_n\) on \(\Sigma\) by giving each \(I \in \MCP_{n,a}(\ugamma)\) weight \(\diam(\Pi(I))^{s_n}\). Then the measures 
\[\nu_n=\frac{1}{n} \sum_{i=0}^{n-1} \eta_n \circ \sigma^{-i}\]
are \(\sigma\)-invariant, ergodic, and satisfy \(\int \phi_i \circ \Pi \diff \nu_n \in (\gamma_i-\varepsilon,\gamma_i+\varepsilon)\) for all \(i \leq k\). Moreover, by Abramov's formula for entropy (see \cite{PU10}, Theorem 2.4.6)
\[h_{\nu_n}=-\frac{1}{n} \sum_{I \in \MCP_{n,a}(\ugamma)} \diam(\Pi(I))^{s_n} \log  \diam(\Pi(I))^{s_n}   \]
and, by Lemma \ref{lem:diameterfunction}, for \(n\) large enough
\[\lambda_{\nu_n} \leq-\frac{1}{n} \sum_{I \in \MCP_{n,a}(\ugamma)} \diam(\Pi(I))^{s_n} \log  \diam(\Pi(I))+\varepsilon.\]
Thus, considering the measures \(\mu_n:= \Pi^* \nu_n\),
\begin{align*}
    \tilde{s}&<s_n+\varepsilon \\
    &\leq \frac{h_{\mu_n}}{\lambda_{\mu_n}-\varepsilon}+\varepsilon \\
    &\leq \frac{\log \zeta}{\log \zeta-\varepsilon}  \frac{h_{\mu_n}}{\lambda_{\mu_n}}+\varepsilon \\
    &\leq \frac{\log \zeta}{\log \zeta-\varepsilon}  \sup_{\mu \in \MCM_{\MCE}(\Lambda,T)} \left\{\frac{h_{\mu}}{\lambda_{\mu}}:\left|\int \phi_i \diff \mu-\gamma_i \right|<\varepsilon, \forall i \leq k \right\}+\varepsilon.
\end{align*}
Taking the limit as \(k \rightarrow \infty\) and \(\varepsilon \rightarrow 0\) completes the proof.
\end{proof}

\subsection{Lower Bound}\label{subsec:approxtheolowerbound}
We now prove the lower bound.
\begin{prop}\label{prop:approxlowerbound}
For \(\ugamma \in Z\) we have
\begin{align*}
\dim \Lambda_{\MCR}(\underline{\gamma}) &\geq \alpha_1(\ugamma)
\end{align*}
\end{prop}

\begin{lemma}\label{lem:approxlowerboundsequencelemma}
Let \((\mu_n)_{n \in \N} \in \MCM(\Lambda,T)\) be a sequence of measures with finite Lyapunov exponent such that the following limits exist
\begin{equation}\label{con2}
    \gamma_i:=\lim_{n \rightarrow \infty} \int \phi_i \diff \mu_n, \; \forall i \in \N.
\end{equation}
Then for \(\underline{\gamma}=(\gamma_i)_{i \in \N}\) we have 
\[\dim \Lambda_{\MCR}({\underline{\gamma}}) \geq  \limsup_{n \rightarrow \infty} \frac{h_{\mu_n}}{\lambda_{\mu_n}}.\]
\end{lemma}

For the proof of this lemma, we use the technique of \(w\)-measures. The term `\(w\)-measures' was introduced in \cite{GR09}, though similar notions had been around before this. For convenience, in this proof we will denote \(f_{\log|T'|}\) by \(f_0\) and \(f_{\phi_i}\) by \(f_i\), where recall these are the corresponding uniformly continuous functions on the shift space. Furthermore, let
\[\Sigma_{\MCR}(\ugamma):=\MCR \cap \{\omega \in \Sigma \setminus \Pi^{-1}(\cup_{j=0}^{\infty} T^{-j}E): \lim_{n \rightarrow \infty} A_n f_i(\omega)=\gamma_i, \forall i \in \N  \}\]
and note that, by (\ref{eqn:birkhoffaverageequality}), \(\Pi(\Sigma_{\MCR}(\ugamma))= \Lambda_{\MCR}(\ugamma).\)

The idea of the proof is to construct a probability measure \(\eta\) which gives mass to \(\Sigma_{\MCR}(\ugamma)\) by defining it on a family of cylinders which has a product structure. We then apply the mass distribution principle to the push-forward measure \(\Pi^* \eta\). Note that by Proposition \ref{prop:entropydensity}, we may assume the measures \(\nu_n:=\mu_n \circ \Pi\) are ergodic and supported on finitely many symbols. For each \(n\), using Birkhoff's ergodic theorem and uniform continuity, by making \(m_n\) large enough we can find a collection of cylinders of length \(m_n\) with total \(\nu_n\)-measure arbitrarily close to one such that each point contained in these cylinders has partial \(m_n\)-Birkhoff average close to \(\int f_i \diff \nu_n \) for all \(0 \leq i \leq n\). By Lemma \ref{lem:diameterfunction} and the fact that the \(\nu_n\) are supported on finitely many symbols, from this we can also approximate the diameter of the corresponding basic intervals.  Furthermore, using the Shannon-McMillan-Breiman theorem we can control the \(\nu_n\)-measure of these cylinders, insisting that their individual masses do not vary too far from \( \exp(- m_n h_{\nu_n})\). The product measure is then constructed by using bridge words to connect each set of cylinders to the cylinders constructed from the next measure in the sequence. Crucially, the length of the bridge words and the convergence constants of the following measure do not depend on our \(m_n\) and so we can choose our \(m_n\) large enough so that any effect from bridging between measures is negligible. Provided the total \(\nu_n\)-measure of each set of cylinders approaches one fast enough, the push-forward of the constructed measure will give mass to \(\Lambda_{\MCR}(\ugamma)\), allowing us to apply the mass distribution principle. 

The proof is split into three parts. We first define the product measure \(\eta\), then show that \(\eta(G)>0\) for some well-behaved set \(G \subset \Sigma_{\MCR}(\ugamma)\), and finally we apply the mass distribution principle.

\begin{proof}[Proof of Lemma \ref{lem:approxlowerboundsequencelemma}]
Applying Proposition \ref{prop:entropydensity} to the measures \(\mu_n \circ \Pi \in \MCM(\Sigma,\sigma)\), we can find ergodic measures \(\nu_n \in \MCM_{\MCE}(\Sigma,\sigma)\) supported on finitely many symbols such that
\[\lim_{n \rightarrow \infty} \frac{h_{\nu_n}}{\lambda_{\nu_n}}=\limsup_{n \rightarrow \infty}  \frac{h_{\mu_n}}{\lambda_{\mu_n}}\]
and 
\[\lim_{n \rightarrow \infty} \int f_i \diff \nu_n= \gamma_i, \: \forall i \in \N.\]
Without loss of generality we may assume that \(\lim_{n \rightarrow \infty} \frac{h_{\nu_n}}{\lambda_{\nu_n}}>0\). This ensures that the constructed measure \(\eta\) does not give mass to individual points, in particular to \(\Pi^{-1}(\cup_{j=0}^{\infty} T^{-j}E) \). 

Let \(k_n \in \N\) be an increasing sequence such that \(\nu_n(\{1,\ldots,k_n\}^{\N})=1\) for all \(n \in \N\). As \(\Sigma\) is mixing there exists \(N_n \in \N\) such that for each pair \(a \in \{1,\ldots,k_n\}^{\N}\), \(b \in\{1,\ldots,k_{n+1}\}^{\N}\) we can choose an admissible word \(w(a,b) \in \MCC_{N_n}(\Sigma)\) connecting \(a\) and \(b\). Furthermore let \(K_n \subset \Sigma\) be an increasing sequence of finite Markov shifts encompassing all elements of \(\Sigma\) consisting only of digits in \(\{1,\ldots,k_{n+1}\}\) and those appearing in the chosen admissible words connecting \(\{1,\ldots, k_n\}\) and \(\{1,\ldots,k_{n+1}\}\). Each \(f_i\), \(0 \leq i \leq n \), is bounded on \(K_n\) by a uniform bound, 
\(\lambda_n \) say. Let \((m_n)_{n \in \N}\) be a increasing sequence of integers such that 
\begin{equation}\label{lambdam}
    \frac{\lambda_{n}}{m_{n}} \rightarrow 0,
\end{equation}
among some other conditions which we will specify in what follows. 

For each \(n \in \N\) and for all cylinders \(u_1 \in \MCC_{m_1}(\Sigma) \cap \{1,\ldots,k_1\}^{m_1},\ldots, u_n \in \MCC_{m_n}(\Sigma) \cap \{1,\ldots,k_n \}^{m_n}\) let
\[\eta([u_1, w_1, u_2 \ldots, w_{n-1}, u_n])=\prod_{i=1}^n \nu_i(u_i),\] 
where 
\[w_i:=(w(u_i|_{m_i},1),1,w(1,u_{i+1}|_1 ))\]
and depends on \(u_i \in \MCC_{m_i}(\Sigma)\) and \(u_{i+1} \in \MCC_{m_{i+1}}(\Sigma)\). We call these \(w_i\) bridge words. They have length \(N_i^*:=2N_i+1\). We use these bridge words to make sure our measure accumulates on \(\MCR\). For all cylinders \(u\) which do not intersect one of these \([u_1, w_1, u_2 \ldots, w_{n-1}, u_n]\), set \(\eta(u)\)=0. Then \(\eta\) is a uniquely defined probability measure supported on 
\[\{ (u_1, w_1, u_2, w_2,\ldots) \in \Sigma :u_i \in \MCC_{m_i}(\Sigma) \cap \{1,\ldots,k_i\}^{m_i}, w_i \text{ bridge word} \}. \]
It is clear that \(\eta(\MCR)=1\).

We now specify the other conditions we want our \(m_n\) to satisfy. Let \((\varepsilon_n)_{n \in \N}\) be a sequence of positive numbers decreasing to zero. Since each \(f_i\) is uniformly continuous, for each \(n \in \N\) there exists \(M_n \in \N\) such that for all \(\omega,\omega'  \in \Sigma\) and for all \(0 \leq i \leq n\)
\begin{equation}\label{eqn:uniformlycontinousapplied}
    d(\omega,\omega' ) \leq 2^{-M_n} \implies |f_i(\omega)-f_i(\omega' )|<\frac{\varepsilon_n}{4}.
\end{equation}
By the Shannon–McMillan–Breiman Theorem for any \(\delta_n>0\) there exists \(M'_n\) such that 
\begin{equation}\label{eqn5}
    \nu_n \left( \left\{ \omega \in \Sigma \cap \{1,\ldots,k_n\}^{\N}: \left|-\frac{1}{n} \log(\nu_n([\omega_1,\ldots,\omega_m]))-h_{\nu_n} \right|<\frac{\varepsilon_{n}}{2}, \forall m \geq M''_n \right\} \right)>1-\delta_n.
\end{equation}
By Birkhoff's ergodic theorem for any \(\delta'_n>0\) there exits \(M''_n\) such that 
\begin{equation}\label{eqn4}
    \nu_n \left( \left\{ \omega \in \Sigma \cap \{1,\ldots,k_n\}^{\N}: \left| A_m f_i(\omega)-\int f_i \diff \nu_n \right|<\frac{\varepsilon_{n}}{2}, \forall 0 \leq i \leq n, \forall m \geq M'_n \right\} \right)>1-\delta'_n
\end{equation}

For \(f_0\) we can further relate the Birkhoff averages to the diameter of basic intervals. By Lemma \ref{lem:diameterfunction} there exists \(M'''_n\) such that for all \(m \geq M'''_n\) and all \(\omega \in K_n \setminus \Pi^{-1}(\cup_{j=0}^{\infty} T^{-j}E)\)
\begin{equation}\label{eqn:diamfunctionapplied}
    \left|A_m(-\log|T' \circ \Pi(\omega)|)-\frac{\log(\diam(C_m(\omega)))}{m} \right| \leq \frac{\varepsilon_n}{2}.
\end{equation}
We let \(M^*_n=\max\{M_n,M'_n,M''_n,M'''_n\}\).

Taking each \(n\) in turn, choose \(m_n\) large enough so that 
\begin{equation}\label{eqn:mncondition}
    \frac{2(N^*_n+M_n+M^*_{n+1}) \lambda_n} {m_n+N^*_n+M^*_{n+1}}<\frac{\varepsilon_n}{4}, 
\end{equation}
\begin{equation}\label{eqn:mncondition2}
    \exp(-(m_n+N^*_n+M^*_{n+1})(h_{\nu_n}-\varepsilon_n))\geq \exp(-m_n(h_{\nu_n}-\varepsilon_n/2))
\end{equation}
and \(m_n \geq M^*_n\). Let \(\MCC_n\) be the collection of \(m_n\)-level cylinders formed by truncating the points in the intersection of the sets in (\ref{eqn5}) and  (\ref{eqn4}). Let 
\[G_n:=\bigcup [u_1,w_1,u_2,w_2,\ldots,w_{n-1},u_n]\]
where the union is over \(u_i \in \MCC_i\) and the \(w_i\) are the bridge words defined earlier. The set \(G:=\cap_{i=1}^\infty G_i \setminus \Pi^{-1}(\cup_{n=0}^{\infty} T^{-n}E) \) satisfies \(\eta(G)>\prod_{n=1}^\infty 1- \delta_n-\delta'_n\), so choosing \(\delta_n,\delta'_n\) small enough we can have \(\eta(G)>0\).

We now show that 
\begin{align}\label{eqn:Gissubset}
    G \ &\subseteq \Sigma_{\MCR}(\underline{\gamma}) \cap \left\{\omega \in \Sigma : \liminf_{n \rightarrow \infty} \frac{-\log(\eta([\omega_1,\ldots,\omega_n]))}{-\log(\diam( C_n(\omega)))} \geq \lim_{n \rightarrow \infty} \frac{h_{\nu_n}}{\lambda_{\nu_n}} \right\} \\
    &=  \Sigma_{\MCR}(\underline{\gamma}) \cap \left\{\omega \in \Sigma : \liminf_{n \rightarrow \infty} \frac{-\log(\eta([\omega_1,\ldots,\omega_n]))}{-\log(\diam( C_n(\omega)))} \geq \limsup_{n \rightarrow \infty}  \frac{h_{\mu_n}}{\lambda_{\mu_n}}\right\} \nonumber.
\end{align}

Given \(\omega \in G\), we may write it in the form 
\[\omega=(u_1,w_1,u_2,w_2,\ldots,w_{n-1},u_n,\ldots) \]
where the bridge words \(w_i\) have length \(N^*_i\). We claim that for each \(k \in \N\) and all \(M_k^*+\sum_{j=1}^{k-1} m_j+N^*_j \leq m \leq M_{k+1}^*+\sum_{j=1}^{k} m_j+N^*_j\) we have
\begin{equation}\label{eqn6}
    \left| A_{1+\sum_{j=1}^{k-1} m_j+N^*_j}^m f_i(\omega) -\int f_i \diff \nu_k \right|<\varepsilon_k, \; \forall 0 \leq i \leq k
\end{equation} 
\begin{equation}\label{eqn7}
    \eta([\omega_1,\ldots,\omega_{m}])<\exp(-(m-\sum_{j=1}^{k-1} m_j+N^*_j)(h_{\nu_k}-\varepsilon_k)) \prod_{i=1}^{k-1} \exp(-(m_i+N^*_i)(h_{\nu_i}-\varepsilon_i))
\end{equation}
\begin{equation}\label{eqn:diameterfunctionepsilon}
 \left|A_m(-\log|T' \circ \Pi(\omega)|)-\frac{\log(\diam (C_m(\omega))}{m} \right|<\varepsilon_k 
\end{equation}
where \(A_l^m f:=\frac{1}{m-l+1}\sum_{i=l-1}^{m-1} f \circ \sigma^i\) is the partial Birkhoff average of \(f\) between terms \(l\) and \(m\). Notice that (\ref{eqn:Gissubset}) follows immediately from these statements. In particular, \(G \subseteq \Sigma_{\MCR}(\ugamma)\) follows directly from (\ref{eqn6}), and the other inclusion can be proved from (\ref{eqn6}), (\ref{eqn7}), (\ref{eqn:diameterfunctionepsilon}) and the elementary fact that if \(a_n,b_n>0\) and \(\sum_{n} b_n= \infty\), then \(\frac{a_n}{b_n} \rightarrow L\) implies \(\frac{\sum_{i=1}^n a_i}{\sum^n_{i=1} b_i} \rightarrow L \). 

Let us prove each of these statements. For the first, we can find
\(\omega'\)
in the set in \((\ref{eqn4})\), with \(n=k\), such that \((\omega'_1,\ldots,\omega'_{m_k})=u_k\). Then by (\ref{eqn:uniformlycontinousapplied}), (\ref{eqn4}) and (\ref{eqn:mncondition}), for any \(i \leq k\)
\begin{align*}
    \left| A_{1+\sum_{j=1}^{k-1} m_j+N^*_j}^m f_i(\omega) -\int f_i \diff \nu_k \right| &\leq \left| A_{1+\sum_{j=1}^{k-1} m_j+N^*_j}^m f_i(\omega) - A_{m-\sum_{j=1}^{k-1} m_j+N^*_j} f_i(\omega') \right| \\
    &+ \left|\int f_i \diff \nu_k-A_{m-\sum_{j=1}^{k-1} m_j+N^*_j} f_i(\omega') \right| \\
    &\leq \frac{m_k-M_k}{m_k+M_{k+1}^*+N^*_k} \left(\frac{\varepsilon_k}{4} \right) + \frac{M_k+N^*_k+M^*_{k+1}}{m_k+M_{k+1}^*+N^*_k}\left(2 \lambda_k+\frac{\varepsilon_k}{4} \right)+ \frac{\varepsilon_k}{2} \\
    &\leq \varepsilon_k.
\end{align*}
The second follows as by (\ref{eqn:mncondition2}), with \(j:=m-\sum_{i=1}^{k-1} m_i+N^*_i\), we have 
\begin{align*}
     \eta([\omega_1,\ldots,\omega_{m}])&=\nu_k((u_k,w_k,u_{k+1})|_1^j)\prod_{i=1}^{k-1} \nu_i(u_i) \\
     &\leq \nu_k((u_k,w_k,u_{k+1})|_1^j)\prod_{i=1}^{k-1} \exp(-m_i(h_{\nu_i}-\varepsilon_i/2) \\
     &\leq \nu_k((u_k,w_k,u_{k+1})|_1^j)\prod_{i=1}^{k-1} \exp(-(m_i+N^*_i)(h_{\nu_i}-\varepsilon_i)).
\end{align*}
Furthermore, if \(M_k^* \leq j \leq m_k\) then \(\nu_k((u_k,w_k,u_{k+1})|_1^j)<\exp(-j(h_{\nu_k}-\varepsilon_k))\) and if \(m_k \leq j \leq m_k+N^*_k+M_{k+1}^* \) then we also have \(\nu_k((u_k,w_k,u_{k+1})|_1^j)<\exp(-m_1(h_{\nu_k}-\varepsilon_k/2))\leq \exp(-j(h_{\nu_k}-\varepsilon_k)) \). 

Finally, for the third we can find \(\omega' \in K_k \setminus \Pi^{-1}(\cup_{n=0}^{\infty} T^{-n}E)\) such that \(\omega'|_1^{m+M_k}=\omega|_1^{m+M_k} \). Then by (\ref{eqn:uniformlycontinousapplied}) and (\ref{eqn:diamfunctionapplied}) we have
\begin{align*}
    \left|A_m(-\log|T' \circ \Pi(\omega)|)-\frac{\log(\diam (C_m(\omega))}{m} \right| &\leq \left|A_m(-\log|T' \circ \Pi(\omega)|)-A_m(-\log|T' \circ \Pi(\omega'))\right| \\
    &+ \left|\frac{\log(\diam (C_m(w))}{m}-A_m(-\log|T' \circ \Pi(\omega')|) \right| \\
    &\leq \varepsilon_k.
\end{align*}

Notice that by (\ref{eqn6}) and (\ref{eqn:diameterfunctionepsilon}) we also have 
\begin{equation}
    \sup_{\omega,\omega'  \in G} \left\{ \left| \frac{1}{m} \log \diam(C_m(\omega))-\frac{1}{m} \log \diam(C_m(\omega' )) \right| \right\} \rightarrow 0.
\end{equation}

We are now ready to apply the mass distribution principle. Given \(\varepsilon>0\), we can find \(M \in \N\) and a subset \(S \subset G\) with \(\eta(S)>0\) such that for all \(m \geq M\) and all \(\omega \in S\)
\begin{equation}\label{eqn:mdp1}
    \sup_{\omega',\omega''  \in G} \left\{ \left| \frac{1}{m} \log \diam(C_m(\omega'))-\frac{1}{m} \log \diam(C_m(\omega'' )) \right| \right\}< \varepsilon,
\end{equation}
\begin{equation}\label{eqn:mdp2}
    \left| \frac{1}{m} \log(\diam(C_m(\omega)))- A_m(-\log|T' \circ \Pi(\omega)|) \right|<\varepsilon,
\end{equation}\label{eqn:mdp3}
\begin{equation}\label{eqn:mdp4}
    \frac{-\frac{1}{m}\log(\eta([\omega_1,\ldots,\omega_m]))}{-\frac{1}{m}\log(\diam (C_m(\omega_1,\ldots,\omega_m)))} \geq \limsup_{n \rightarrow \infty} \frac{h_{\mu_n}}{\lambda_{\mu_n}}-\varepsilon
\end{equation}
and 
\begin{equation}\label{eqn:mdp5}
    \frac{\lambda_{k(m)}+\varepsilon_{k(m)+1}}{m}<\varepsilon, 
\end{equation}
where \(k(m)\) is the unique \(k \in \N\) such that \(M_k^*+\sum_{i=1}^{k-1} m_i+N^*_i \leq m < M_{k+1}^*+\sum_{i=1}^{k} m_i+N^*_i\).

Let \(0<r<\sup\{\diam(C_M(\omega)):\omega \in S\} \). Let \(m\) be the smallest positive integer such that 
\(\diam( C_m(\omega))<r\)
for all \(\omega \in S\). Clearly \(m \geq M+1\). We also have that 
\begin{equation}\label{eqn:mbound}
    m \leq \frac{-\log r}{\log \zeta}+1 
\end{equation}
since by the definition of \(m\) there exists \(\omega'  \in S\) such that 
\[r \leq \diam( C_{m-1}(\omega' )) \leq 1/\zeta^{m-1}.\]
For this \(\omega'  \in S\), we must have 
\[r\leq \diam( C_{m-1}(\omega' ))<\exp \left( S_{m-1}(-\log|T' \circ \Pi(\omega' )|)+(m-1)\varepsilon \right)\]
and so for all \(\omega \in S\)
\begin{align*}
    \diam( C_m(\omega)) &\geq e^{-\varepsilon m} \diam (C_m(\omega' )) \\
    &>\exp(S_{m}(-\log|T' \circ \Pi(\omega' )|)-2m\varepsilon) \\
    &>\exp(S_{m-1}(-\log|T' \circ \Pi(\omega' )|)-3m\varepsilon)  \\
    &>r \exp(-4m \varepsilon),
\end{align*}
where the third inequality follows from (\ref{eqn:uniformlycontinousapplied}) and (\ref{eqn:mdp5}) (to see this, note that we can find \(\omega'' \in K_{k(m)} \setminus \Pi^{-1}(\cup_{n=0}^{\infty} T^{-n}E)\) such that \(d(\sigma^{m-1}\omega',\omega'') \leq 2^{-m_{k(m)+1}} \leq 2^{-M_{k(m)+1}}\)). Hence for any \(x \in [0,1]\), the ball \(B_r(x)\)  overlaps \(\Pi(S)\) at no more than \(\exp(4m \varepsilon)+2\) \(m\)th level intervals. This is less than  \(\exp(5m \varepsilon)\) provided our \(M\) was chosen large enough. Thus by equations (\ref{eqn:mdp4}) and (\ref{eqn:mbound})
\begin{align*}
    0<\Pi^* \eta (B_r(x) \cap \Pi(S)) &\leq \exp(5m \varepsilon) r^{\limsup_{n \rightarrow \infty} \frac{h_{\mu_n}}{\lambda_{\mu_n}} -\varepsilon} \\
    &\leq e^{5 \varepsilon} r^{\limsup_{n \rightarrow \infty} \frac{h_{\mu_n}}{\lambda_{\mu_n}} -\varepsilon-\frac{5\varepsilon}{\log \zeta}}.
\end{align*}
Therefore by the mass distribution principle 
\[\dim \Lambda_{\MCR}(\ugamma) \geq \dim \Pi(S) \geq \limsup_{n \rightarrow \infty} \frac{h_{\mu_n}}{\lambda_{\mu_n}} -\varepsilon-\frac{5\varepsilon}{\log \zeta}\]
and letting \(\varepsilon \rightarrow 0\) finishes the proof.
\end{proof}

\begin{remark}
If desired, we could have instead used the bridge words \[w_i:=(w(u_i|_{m_i},u_1|_{1}),(u_1,w_1,\ldots,u_i)|_1^i,w((u_1,w_1,\ldots,u_i)_i,u_{i+1}|_1 ))\]
which depend on \(u_1 \in \MCC_{m_1}(\Sigma)\),\ldots,\(u_{i+1} \in \MCC_{m_{i+1}}(\Sigma)\) and have length \(2N_i+i\). This would make the measure accumulate on the set
\[\MCR'(\Sigma):=\{\omega \in \Sigma : \exists \: n_k \rightarrow \infty \text{ s.t. } \lim_{k \rightarrow \infty} \sigma^{n_k}(\omega)=\omega\} \]
which is a subset of \(\MCR\). Therefore, in Theorem \ref{theo:uiapprox} and Theorem \ref{theo:uiexact} we could alternatively take the intersection with the set
\[\MCR'(\Lambda):=\{x \in \Lambda \setminus \cup_{j=0}^{\infty} T^{-j}E: \exists \: n_k \rightarrow \infty \text{ s.t. } \lim_{k \rightarrow \infty} T^{n_k}(x)=x\} \]
since \(\MCR'(\Lambda) \cap \Lambda(\underline{\gamma})=\Lambda_{\MCR'(\Sigma)}(\underline{\gamma})\). 
Note that this is a commonly used notion of recurrence (for example, the one considered in \cite{Iom05}).
\end{remark}

\section{Proof of Theorem \ref{theo:uiexact}}\label{sec:uiexact}
In this section we assume that the \(\phi_i\) are also bounded and 
\[\inf\{|T'(x)|:x \in \overline{I_i} \setminus E\} \xrightarrow[i \rightarrow \infty]{} \infty.\] Before proving Theorem \ref{theo:uiexact}, we first prove Lemma \ref{lem:sinfinitycharacterisation} whose proof was deferred when proving Proposition \ref{prop:uppersemicontinuity}. This follows immediately from the following two lemmas. 

\begin{lemma}\label{lem:approxsinfinity}
There exists a sequence of measures \(\mu_n \in \MCM(\Lambda,T)\) such that
\[\lim_{n \rightarrow \infty} \lambda_{\mu_n}=\infty, \: \lim_{n \rightarrow \infty} \frac{h_{\mu_n}}{\lambda_{\mu_n}}=s_{\infty}\]
\end{lemma}

\begin{proof}
We adapt the proof of Lemma 2.5 in \cite{FJLR15} into our setting. First suppose \(s_{\infty}>0\). Let \(\varepsilon_n\) be a positive sequence converging to zero. Note that for any \(\mu \in \MCM(\Lambda,T)\) such that \(\frac{h_{\mu}}{\lambda_{\mu}} \geq s_{\infty}+2\varepsilon_n\) we have 
\begin{align*}
    P(-(s_{\infty}+\varepsilon_n)\log|T'|) &\geq h_{\mu}-(s_{\infty}+\varepsilon_n) \lambda_{\mu} \\
    &\geq \varepsilon_n \lambda_{\mu},
\end{align*}
so \(\lambda_{\mu} \leq \frac{P(-(s_{\infty}+\varepsilon_n)\log|T'|)}{\varepsilon_n}\).
Now take two positive sequences \((t_n)_{n \in \N}\) and \((k_n)_{n \in \N}\) such that for each \(n\), \(t_n<s_{\infty}\), \(\lim_{n \rightarrow \infty} t_n=s_{\infty}\) and \(\lim_{k \rightarrow \infty}k_n =\infty \). Since for all \(n \in \N\) we have \(P(-t_n \log|T'|)=\infty\), we can find a sequence of measures \(\mu_n \in \MCM(\Lambda,T)\) such that  \(h_{\mu_n}-t_n \lambda_{\mu_n}>k_n\)
and hence \(\frac{h_{\mu_n}}{\lambda_{\mu_n}}>t_n\). Furthermore, by the fact that \(\lambda_{\mu_n} \geq h_{\mu_n}\) (see Remark \ref{remark:entropylessthanlypunov}), we have \(\lambda_{\mu_n} > k_n\). However, if we assume \(k_n \geq \frac{P(-(s_{\infty}+\varepsilon_n)\log|T'|)}{\varepsilon_n} \), then \(\frac{h_{\mu_n}}{\lambda_{\mu_n}} \leq s_{\infty}+2\varepsilon_n\). So,
\[\lim_{n \rightarrow \infty} \frac{h_{\mu_n}}{\lambda_{\mu_n}}=s_{\infty}.\]

The case when \(s_{\infty}=0\) and \(h_{\text{top}}(\Sigma,\sigma)=\infty\) can be proved similarly by taking measures \(\mu_n\) such that \(h_{\mu_n} \rightarrow \infty\). If \(h_{\text{top}}(\Sigma,\sigma)<\infty\), then \(s_{\infty}=0\) and the result follows simply as we can find a sequences of measures \(\mu_n\) such that \[\lim_{n \rightarrow \infty} \lambda_{\mu_n}=\infty\] by our assumption that 
\[\inf\{|T'(x)|:x \in \overline{I_i}\} \rightarrow \infty \]
and Lemma 4.16 in \cite{IV19}.
\end{proof}

\begin{lemma}\label{lem:boundedlyapunov}
For any \(\delta>0\) there is \(K(\delta)>0\) such that if \(\mu \in \MCM(\Lambda,T)\) and \(\frac{h_{\mu}}{\lambda_{\mu}}>s_{\infty}+\delta\) then \(h_{\mu} \leq \lambda_{\mu} \leq K(\delta)\).
\end{lemma}
\begin{proof}
This follows in the same way as in the proof of Lemma 6.4 in \cite{FJLR15}. Let \(t \in \R\) be such that \(s_{\infty}<t<s_{\infty}+\delta\). By the variational principle we have \(h_{\mu}-t\lambda_{\mu} \leq P(-t\log|T'|)\). Since  \(\frac{h_{\mu}}{\lambda_{\mu}}>s_{\infty}+\delta\), we have 
\[P(-t\log|T'|) \geq (s_{\infty}+\delta-t)\lambda_{\mu}.\]
So, 
\[\lambda_{\mu} \leq \frac{P(-t\log|T'|)}{s_{\infty}+\delta-t}.\]
\end{proof}

The next lemma shows that \(\alpha_1(\ugamma) \geq s_{\infty}\)  for any \(\ugamma \in Z\). As Theorem \ref{theo:uiapprox} is already proved, from this it follows that 
\[\dim \Lambda_{\MCR}(\ugamma)=\max\{s_{\infty},\alpha_1(\ugamma)\}. \]
Theorem \ref{theo:uiexact} will then follow if we show that \(\alpha_1(\ugamma)>s_{\infty}\) implies \(\alpha_1(\ugamma)=\alpha_3(\ugamma)\). Note that \(\alpha_1(\ugamma) \geq \alpha_3(\ugamma)\) is clear. To show the other inequality, we apply Theorem \ref{theo:uiapprox} to get a sequence of measures \(\mu_n \in \MCM(\Lambda,T)\) such that \(\lambda_{\mu_n}<\infty\), \(\frac{h_{\mu_n}}{\lambda_{\mu_n}} \rightarrow \alpha_1(\ugamma) \)
and
\[\int \phi_i \diff \mu_n \rightarrow \gamma_i, \: \forall i \in \N.\]
However, the inequality  \(\alpha_1(\ugamma) \leq \alpha_3(\ugamma)\)  then follows immediately from Proposition \ref{prop:uppersemicontinuity}.  Hence to finish the proof of Theorem \ref{theo:uiexact}, we only need to prove the following lemma.

\begin{lemma}\label{sinfinityisabound}
Let \(\ugamma \in Z\), \(k \in \N\), \(\varepsilon>0\) and \(\mu \in \MCM(\Lambda,T)\) be such that 
\[\lambda_{\mu}<\infty, \qquad \sup_{1 \leq i \leq k} \left|\int \phi_i \diff \mu-\gamma_i \right|\leq \varepsilon.\]
Then there exists a measure \(\nu \in \MCM(\Lambda,T)\) such that 
\[\frac{h_{\nu}}{\lambda_{\nu}} \geq s_{\infty}-\varepsilon, \qquad \sup_{1 \leq i \leq k} \left|\int \phi_i \diff\nu-\gamma_i \right|\leq 3\varepsilon.\]
\end{lemma}

\begin{proof}
The proof follows as in Lemma 5.1 in \cite{FJLR15}. Let \(\varepsilon>0\) and let \(A=\sup_{1 \leq i \leq k} \sup_{x \in \Lambda} |\phi_i(x)|\). By Lemma \ref{lem:approxsinfinity} we can find a sequence of measures \(\mu_n \in \MCM(\Lambda,T)\) such that \(\lim_{n \rightarrow \infty} \lambda_{\mu_n}=\infty\) and \(\lim_{n \rightarrow \infty} \frac{h_{\mu_n}}{\lambda_{\mu_n}}=s_{\infty}\). Consider the measure 
\[\nu_n=\left(1-\frac{\varepsilon}{A} \right)\mu+\frac{\varepsilon}{A}\mu_n.\]
Then we have that for each \(1 \leq i \leq k\) 
\[\left| \int \phi_i \diff \nu_n-\gamma_i \right| \leq \left|\int \phi_i \diff \mu-\gamma_i \right|+\left| \int \phi_i \diff \mu-\int \phi_i \diff \nu_n \right| \leq 3 \varepsilon. \]
Furthermore 
\begin{align*}
    \lim_{n \rightarrow \infty} \frac{h_{\nu_n}}{\lambda_{\nu_n}} &= \lim_{n \rightarrow \infty} \frac{(1-\varepsilon/A)h_{\mu}+\varepsilon/A h_{\mu_n}}{(1-\varepsilon/A)\lambda_{\mu}+\varepsilon/A \lambda_{\mu_n}} \\
    &= \lim_{n \rightarrow \infty} \frac{h_{\mu_n}}{\lambda_{\mu_n}} \\
    &= s_{\infty}.
\end{align*}
\end{proof}

\section{Proof of Theorem \ref{theo:uiexactbounded}}
Throughout this section \((\phi_i)_{i \in \N}\) refers to a sequence of functions in \(C_0(\Lambda)\) and \(\log|T'|-L \in C_0(\Lambda)\) for some \(L \geq \log \zeta\) . Notice if \(x \in \Lambda_{\MCT}\) then \(\lim_{n \rightarrow \infty} A_n \phi(x)=0\) for all \(\phi \in C_0(\Lambda)\). Therefore for any \(\ugamma \not=\underline{0}\) we must have \(\Lambda_{\MCR}(\ugamma)  =\Lambda(\ugamma)\). Furthermore note that \(\Lambda(\underline{0}) =\Lambda_{\MCR}(\underline{0}) \cup \Lambda_{\MCT}\). Hence to prove Theorem \ref{theo:uiexactbounded} it suffices to prove the following lemma. 

\begin{lemma}\label{prop6}
Let \((\phi_i)_{i \in \N} \subset C_0(\Lambda)\) be a sequence of functions, then for \(\ugamma \in Z\)
\[\alpha_1(\ugamma)=\alpha_4(\ugamma).\]
Moreover,
\[Z=\left\{ \underline{\gamma} \in \R^{\N}: \exists \mu \in \MCM_{\leq 1}(\Lambda,T), \forall i \in \N, \int \phi_i \diff \mu=\gamma_i \right\}.\]
\end{lemma}

\begin{proof}
Let \(\ugamma \in Z\). We first prove the inequality \(\alpha_1(\ugamma) \leq \alpha_4(\ugamma)\). Let \(\varepsilon_n\) be a positive sequence converging to zero. We may find a sequence of measures \(\mu_n \in \MCM(\Lambda,T)\) such that 
\begin{equation}\label{eqn13}
    \left|\int \phi_i \diff \mu_n-\gamma_i \right|<\varepsilon_n, \forall i \leq n
\end{equation}
and
\[\lim_{n \rightarrow \infty} \frac{h_{\mu_n}}{\lambda_{\mu_n}} = \alpha_1(\ugamma).\]
Applying Theorem 1.2 in \cite{IV19} to the measures \(\nu_n:=\mu_n \circ \Pi\), there exists a measure \(\nu \in \MCM_{\leq 1}(\Sigma,\sigma)\) and a subsequence \((\nu_{n_j})_{j \in \N}\) such that \(\nu_{n_j}\) converges to \(\nu\) in the topology of convergence of cylinders. Without loss of generality we may assume that \(\nu_{n}\) converges to this measure. Let \(\mu:=\Pi^* \nu\). By Lemma \ref{lem:convergenceoncylinderstestfunctions} and (\ref{eqn13}) we have
\[\int \phi_i \diff \mu=\int \phi_i \circ \Pi \: \diff \nu=\lim_{n \rightarrow \infty} \int \phi_i \circ \Pi \: \diff \nu_n =\lim_{n \rightarrow \infty} \int \phi_i \diff \mu_n= \gamma_i , \forall i \in \N.\]
This shows that 
\[Z \subseteq \left\{ \underline{\gamma} \in \R^{\N}: \exists \mu \in \MCM_{\leq 1}(\Lambda,T), \forall i \in \N, \int \phi_i \diff \mu=\gamma_i \right\}.\]
Furthermore, since \(\log|T' \circ \Pi|-L \in C_0(\Lambda) \) and by Theorem \ref{theo:uppersemicontiuitymike}, 
\begin{align*}
    \alpha_1(\ugamma) &=\lim_{n \rightarrow \infty} \frac{h_{\mu_n}}{\lambda_{\mu_n}} \\
    &= \lim_{n \rightarrow \infty} \frac{h_{\nu_n}}{\int \log|T' \circ \Pi|-L \diff \nu_n+L} \\
    &\leq \frac{ |\nu|h_{\frac{\nu}{|\nu|}}+(1-|\nu|) \delta_{\infty} }{\int \log|T' \circ \Pi|-L \diff \nu+L} \\
    &= \frac{ |\mu|h_{\frac{\mu}{|\mu|}}+(1-|\mu|) \delta_{\infty} }{|\mu|\lambda_{\frac{\mu}{|\mu|}}+(1-|\mu|)L} \\
    & \leq \sup_{\mu \in \MCM_{\leq 1}(\Lambda,T)} \left\{\frac{ |\mu|h_{\frac{\mu}{|\mu|}}+(1-|\mu|) \delta_{\infty} }{|\mu|\lambda_{\frac{\mu}{|\mu|}}+(1-|\mu|)L}: \int \phi_i \diff \mu =\gamma_i, \forall i \in \N \right\} \\ 
    &= \alpha_4(\ugamma).
\end{align*} 

We now prove the inequality \(\alpha_4(\ugamma) \leq \alpha_1(\ugamma)\). Let \(\eta>0\) and choose \(\mu \in \MCM_{\leq 1}(\Lambda,T)\) such that
\[ \frac{ |\mu|h_{\mu / |\mu|}+(1-|\mu|) \delta_{\infty} }{|\mu|\lambda_{\frac{\mu}{|\mu|}}+(1-|\mu|)L} \geq \alpha_4(\ugamma)  - \eta \]
and 
\[\int \phi_i \diff \mu =\gamma_i, \forall i \in \N. \]
We can find a sequence of measures \(\nu_n \in \MCM(\Sigma,\sigma)\) converging to the zero measure on cylinders such that \(h_{\nu_n} \rightarrow \delta_{\infty}\). Consider the measures \(\mu_n=\mu+ (1-|\mu|)\Pi^*\nu_n\). Notice that \( \mu_n \circ \Pi\) converges to \(\mu \circ \Pi\) on cylinders, therefore
\begin{equation}\label{eqn12}
    \int \phi_i \diff \mu_n = \int \phi_i \circ \Pi \: d(\mu_n \circ \Pi) \xrightarrow[n \to \infty]{} \int \phi_i \circ \Pi \: d(\mu \circ \Pi) = \int \phi_i \diff \mu= \gamma_i, \forall i \in \N.
\end{equation}
This shows that
\[Z \supseteq \left\{ \underline{\gamma} \in \R^{\N}: \exists \mu \in \MCM_{\leq 1}(\Lambda,T), \forall i \in \N, \int \phi_i \diff \mu=\gamma_i \right\}.\]
By the affinity of the entropy map (see \cite[Theorem 8.1]{Wal81}, noting it also holds in this non-compact setting), we have
\[\frac{h_{\mu_n}}{\lambda_{\mu_n}} = \frac{|\mu|h_{\frac{\mu}{ |\mu|}}+(1-|\mu|) h_{\nu_n}}{|\mu| \lambda_{\frac{\mu}{|\mu|}}+(1-|\mu|)\int \log |T' \circ \Pi|-L \diff \nu_n+(1-|\mu|)L}, \forall n \in \N. \]
Hence as \(h_{\nu_n} \rightarrow \delta_{\infty}\) and \(\int \log|T' \circ \Pi|-L \diff \nu_n \rightarrow 0\), for all \(n\) sufficiently large we have 
\[ \frac{h_{\mu_n}}{\lambda_{\mu_n}} \geq  \frac{|\mu|h_{\frac{\mu}{|\mu|}}+(1-|\mu|) \delta_{\infty}}{|\mu| \lambda_{\frac{\mu}{|\mu|}}+(1-|\mu|)L}-\eta.\]
Moreover, by (\ref{eqn12}), for any \(\varepsilon>0\) and \(k \in \N\) and all \(n\) sufficiently large we have
\[\left| \int \phi_i \diff \mu_n - \gamma_i \right|<\varepsilon , \forall i \leq k.\]
We may choose \(\varepsilon>0\) small enough and \(k \in \N\) large enough so that 
\[\sup_{\mu \in \MCM(\Lambda,T)} \left\{\frac{h_{\mu}}{\lambda_{\mu}}:\left|\int \phi_i \diff \mu-\gamma_i \right|<\varepsilon, \forall i \leq k, \: \lambda_{\mu}<\infty \right\} \leq \alpha_1(\ugamma)+\eta.\]
Putting this together, with \(n\) large enough we have
\begin{align*}
    \alpha_4(\ugamma) &\leq  \frac{ |\mu|h_{\frac{\mu}{ |\mu|}
    }+(1-|\mu|) \delta_{\infty} }{|\mu|\lambda_{\frac{\mu}{|\mu|}}+(1-|\mu|)L}  +\eta\\
    & \leq \frac{h_{\mu_n}}{\lambda_{\mu_n}} +2\eta\\
    & \leq \sup_{\mu \in \MCM(\Lambda,T)} \left\{\frac{h_{\mu}}{\lambda_{\mu}}:\left|\int \phi_i \diff \mu-\gamma_i \right|<\varepsilon, \forall i \leq k, \: \lambda_{\mu}<\infty \right\} +2\eta\\
    & \leq \alpha_1(\ugamma) + 3\eta.
\end{align*}
This concludes the proof of Lemma \ref{prop6}.
\end{proof}

\section{Applications}\label{sec:applications}
\subsection{Frequency of digits}
Our first application of our theorems is to the frequency of digit case, when \(\phi_i=1_{I_i}\). In this case 
\[\Lambda(\underline{\gamma}):=\left\{ x \in \Lambda \setminus \cup_{k=0}^{\infty} T^{-k}E:\lim_{n \rightarrow \infty} \frac{1}{n} \# \{j \leq n: T^j(x) \in I_i \} =\gamma_i \text{ for all } i \in \N \right\}.\]
Here for all \(\ugamma \in Z \setminus \{\underline{0}\}\) we have \(\Lambda(\underline{\gamma})=\Lambda_{\MCR}(\underline{\gamma})\), and \(\dim \Lambda(\underline{0})=\max \{\dim \Lambda_{\MCR}(\underline{0}),\dim \Lambda_{\MCT} \}\). Note that it is possible that \(\Lambda(\underline{0})= \emptyset\), for example if \(\Sigma\) is as in Example 4.17 in \cite{IV19}. 
The maps \(T_i:I_i \rightarrow [0,1]\), where  \(I_i=(1/2^{i},1/2^{i-1}]\), given by
\[T_1(x)=2x-1\]
\[T_i(x)=2^{i-1}x, \: \forall i \geq 2\]
give such a \(\Sigma\).

By Theorem \ref{theo:uiexact}, if 
\begin{equation*}\label{eqn:logTlimitapplication}
    \inf\{|T'(x)|:x \in \overline{I_i} \setminus E\} \rightarrow \infty
\end{equation*}
then for \(\ugamma \in Z_0\) 
\[\dim \Lambda(\ugamma)= \max \left\{s_{\infty},\sup_{\mu \in \MCM(\Lambda,T)} \left\{\frac{h_{\mu}}{\lambda_{\mu}}:\mu(I_i) =\gamma_i, \forall i \in \N, \lambda_{\mu}<\infty \right\} \right\},  \] 
and for \(\ugamma \in Z \setminus (Z_0 \cup \{0\})\), \(\dim \Lambda(\ugamma)=s_{\infty}\). Moreover, if \(\underline{0} \in Z\), then it satisfies
\[\dim \Lambda(\underline{0})=\max\{s_{\infty},\dim \Lambda_{\MCT}\}.\]

In the case where \(\log|T'|\) is as in Theorem \ref{theo:uiexactbounded}, for all 
\[\ugamma \in \left\{ \underline{\gamma}' \in \R^{\N}: \exists \mu \in \MCM_{\leq 1}(\Lambda,T), \forall i \in \N, \mu(I_i)=\gamma'_i \right\}\setminus \{\underline{0}\} \]
we have 
\[\dim \Lambda(\ugamma)= \sup_{\mu \in \MCM_{\leq 1}(\Lambda,T)} \left\{\frac{|\mu|h_{\frac{\mu}{|\mu|}}+(1-|\mu|)\delta_{\infty}}{|\mu|\lambda_{\frac{\mu}{|\mu|}}+(1-|\mu|)L}:\mu(I_i)=\gamma_i, \forall i \in \N  \right\}.\]
Furthermore,
\[\dim \Lambda(\underline{0})=\max \left\{ \frac{\delta_{\infty}}{L},\dim \Lambda_{\MCT} \right\}.\] 

A major difference between the two cases is that if \(\log|T'|\) is as in Theorem \ref{theo:uiexact}, then when \(\ugamma \in Z\) and \(\sum_{i=1}^{\infty} \gamma_i <1\), \(\dim \Lambda(\ugamma)\) can only take one value, whereas if \(\log|T'|\) is as in Theorem \ref{theo:uiexactbounded} then it can take a range of values. 

\subsection{A countably piecewise linear, uniformly expanding map}
In \cite{SV97}, \cite{BT12}, \cite{BT15} and \cite{IJT17} they consider the map \(F_{\lambda}:(0,1] \rightarrow (0,1]\), where \(\lambda \in (0,1)\), defined by
\begin{figure}[h!]
\begin{minipage}[c]{\textwidth-7cm}
      \[F_{\lambda}(x):=
  \begin{cases}
    \frac{x-\lambda}{1-\lambda}, & \text{for } x \in I_1 \\
    \frac{x-\lambda^n}{\lambda(1-\lambda)}, & \text{for } x \in I_n, n \geq 2,
  \end{cases}
\]
for intervals \(I_n:=(\lambda^n,\lambda^{n-1}]\). 
   \end{minipage}
   \begin{minipage}[c]{3cm}
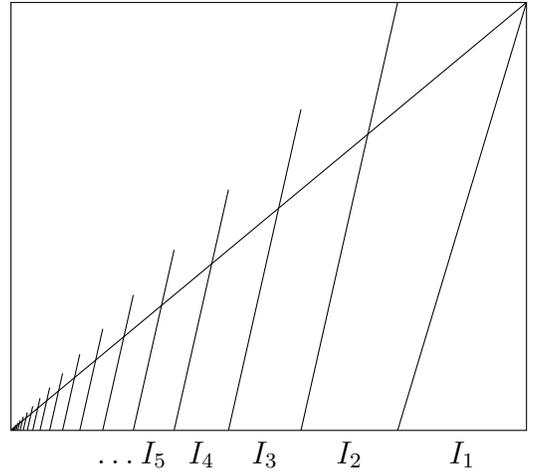
    \begin{tikzpicture}[scale=1]
        \begin{axis}[ymin=0,ymax=1,xmax=1,xmin=0,xtick=\empty,xticklabels=\empty, ,ytick=\empty, xtick style={draw=none}, extra x ticks={(1-0.75)/2+0.75,(0.75-0.75^2)/2+0.75^2,(0.75^2-0.75^3)/2+0.75^3,(0.75^3-0.75^4)/2+0.75^4,(0.75^4-0.75^5)/2+0.75^5,(0.75^5-0.75^6)/2+0.75^6},
    extra x tick labels={$I_1$,$I_2$,$I_3$,$I_4$, $I_5$,$\ldots\vphantom{I_6} $}]
    \addplot[smooth,samples=200,domain=0.75:1]{(x-0.75)/(0.25)};
    \addplot[smooth,samples=200,domain=0:1]{x};
    \foreach \n in {2, ...,30}
    {\addplot[smooth,samples=200,domain=0.75^\n:0.75^(\n-1)]{(x-0.75^\n)/(0.75*0.25)};}
    \end{axis}
    \end{tikzpicture}
   \end{minipage}
\end{figure}

This map was first proposed by van Strien to Stratmann as a model for an induced map of a Fibonacci unimodal map. It is easy to see that it can be coded by a topologically mixing Markov shift and that \(\log|F_{\lambda}'|+\log(\lambda(1-\lambda))\in C_0(\Lambda)\). Thus we can apply Theorem \ref{theo:uiexactbounded} directly. Given a sequence of functions \(\phi_i \in C_0(\Lambda)\) we have for \(\ugamma \in Z \setminus \{\underline{0}\}\)
\[\dim \Lambda (\ugamma)= \alpha_4(\ugamma) \]
and 
\[\dim \Lambda (\underline{0})=\max\{\alpha_4(\underline{0}),\dim \Lambda_{\MCT} \}.\]
In fact, we can find more explicit expressions for the dimensions. Theorem 1 in \cite{SV97} says that
\[\dim \Lambda_{\MCT}=\begin{cases}
  -\frac{\log 4}{\log(\lambda(1-\lambda))} &, \text{for } \lambda \leq \frac{1}{2}  \\
  1 &, \text{for } \lambda \geq \frac{1}{2}.
\end{cases} \]
Moreover, they show in \cite{BT12} (Theorem B) that \(h_{\text{top}}(\Sigma,\sigma)=\log 4 \). We have the following lemma.
\begin{lemma}
\(\delta_{\infty}(\Sigma,\sigma)=h_{\text{top}}(\Sigma,\sigma)=\log 4\).
\end{lemma}

\begin{proof}
Let \(\varepsilon>0\) and let \(\nu \in \MCM(\Sigma,\sigma)\) be such that \(h_{\nu} \geq h_{\text{top}}(\Sigma,\sigma)-\varepsilon\). We can define measures \(\nu_n \in \MCM(\Sigma,\sigma)\) by letting 
\(\nu_n([\omega_1,\ldots,\omega_k]):=\nu([\omega_1-n,\ldots, \omega_k-n])\)
for all \(k \in \N\) and all cylinders \([\omega_1,\ldots,\omega_k] \in \MCC_k(\Sigma)\) such that \(\omega_j \geq n+1\) for all \(1 \leq j \leq k\). For all other cylinders we set \(\nu_n(C)=0\). Then \(h_{\nu_n}=h_{\nu}>h_{\text{top}}(\Sigma,\sigma)-\varepsilon\) and \(\nu_n \rightarrow 0\) on cylinders.
\end{proof}

Thus, for \(\ugamma \in Z \setminus \{\underline{0}\}\)
\[\dim \Lambda (\ugamma)=  \sup_{\mu \in \MCM_{\leq 1}(\Lambda,T)} \left\{\frac{|\mu|h_{\frac{\mu}{|\mu|}}+(1-|\mu|)\log 4}{|\mu|\lambda_{\frac{\mu}{|\mu|}}-(1-|\mu|)\log(\lambda(1-\lambda))}:\int \phi_i \diff \mu=\gamma_i, \forall i \in \N  \right\}.\]
Moreover, if \(\lambda \geq \frac{1}{2}\) we have
\[\dim \Lambda (\underline{0})=1 \]
and if \(\lambda < \frac{1}{2}\),
\[\dim \Lambda (\underline{0})= \sup_{\mu \in \MCM_{\leq 1}(\Lambda,T)} \left\{\frac{|\mu|h_{\frac{\mu}{|\mu|}}+(1-|\mu|)\log 4}{|\mu|\lambda_{\frac{\mu}{|\mu|}}-(1-|\mu|)\log(\lambda(1-\lambda))}:\int \phi_i \diff \mu=0, \forall i \in \N  \right\}.\]

\subsection{Cases when \(\alpha_4(\ugamma)=\alpha_3(\ugamma)\)}
It is natural to ask whether the supremum \(\alpha_4(\ugamma)\) in Theorem \ref{theo:uiexactbounded} can instead be written as a supremum over probability measures; that is, when does \(\alpha_4(\ugamma)=\alpha_3(\ugamma)\)? Notice that this will hold in the frequency of digits case whenever the sum of the frequencies is equal to one. Corollary 2.4 in \cite{IJT17} gives other circumstances where this is the case. Note that when we have only one function \(\phi:\Lambda \rightarrow \R\), the set 
\[\left\{ \gamma \in \R^{\N}: \exists \mu \in \MCM(\Lambda,T), \int \phi_i \diff \mu=\gamma \right\}\]
is an interval, with end points \(\gamma_m, \gamma_M\) say. An application of their results, which they prove much more generally, says that when we have only one function \(\phi \in C_0(\Lambda)\) and both \(f_\phi\), \(f_{\log|T'|}\) are moreover locally H\"{o}lder continuous, then for \(\gamma \in (\gamma_m,\gamma_M)\)
\[\dim \Lambda(\gamma)=\sup_{\mu \in \MCM(\Lambda,T)} \left\{\frac{h_{\mu}}{\lambda_{\mu}}:\int \phi \diff \mu=\gamma \right\}. \]
We prove the following theorem which says that this holds whenever we have finitely many functions, and that it is not necessary to also assume that \(f_{\phi_i}\) and \(f_{\log|T'|}\) are locally H\"{o}lder continuous. 

\begin{theo}
In the setting of Theorem \ref{theo:uiexactbounded}, let \(\phi_1,\ldots,\phi_k \in C_0(\Lambda)\) and suppose that \(\{\int \phi_i \diff \mu:\mu \in \MCM(\Lambda,T)\} \) is not a singleton for all \(i \in \N\). Then for \(\gamma \in \interior Z_0^{(k)}\) 
\begin{equation*}\label{eqn:alpha3equalalpha4}
    \sup_{\mu \in \MCM_{\leq 1}(\Lambda,T)} \left\{\frac{|\mu|h_{\frac{\mu}{|\mu|}}+(1-|\mu|)\delta_{\infty}}{|\mu|\lambda_{\frac{\mu}{|\mu|}}+(1-|\mu|)L}:\int \phi_i \diff \mu=\gamma_i, 1 \leq i \leq  k  \right\}=\sup_{\mu \in \MCM(\Lambda,T)} \left\{\frac{h_{\mu}}{\lambda_{\mu}}:\int \phi_i \diff \mu =\gamma_i, 1 \leq i \leq k \right\}, 
\end{equation*}
where
\[Z_0^{(k)}=\left\{ \underline{\gamma} \in \R^{k}: \exists \mu \in \MCM(\Lambda,T), \int \phi_i \diff \mu=\gamma_i, \forall 1 \leq i \leq k \right\}. \]
\end{theo}

\begin{proof}
The inequality
\[ \sup_{\mu \in \MCM_{\leq 1}(\Lambda,T)} \left\{\frac{|\mu|h_{\frac{\mu}{|\mu|}}+(1-|\mu|)\delta_{\infty}}{|\mu|\lambda_{\frac{\mu}{|\mu|}}+(1-|\mu|)L}:\int \phi_i \diff \mu=\gamma_i, 1 \leq i \leq  k  \right\} \geq \sup_{\mu \in \MCM(\Lambda,T)} \left\{\frac{h_{\mu}}{\lambda_{\mu}}:\int \phi_i \diff \mu =\gamma_i, 1 \leq i \leq k \right\}\]
is clear. We prove the other inequality for \(\ugamma \in \interior Z_0^{(k)} \setminus \{\underline{0}\}\). The case when \(\ugamma=\underline{0}\) can be proved similarly. Choose \(\ugamma \in \interior Z_0^{(k)} \setminus \{\underline{0}\}\) and let \(\varepsilon>0\) be such that the open ball \(B(\ugamma,\varepsilon)\) is a subset of the interior of \(Z_0^{(k)}\). We may find \(\mu \in \MCM_{\leq 1}(\Lambda,T)\) such that 
\[\frac{|\mu|h_{\frac{\mu}{|\mu|}}+(1-|\mu|)\delta_{\infty}}{|\mu|\lambda_{\frac{\mu}{|\mu|}}+(1-|\mu|)L} \geq \sup_{\mu \in \MCM_{\leq 1}(\Lambda,T)} \left\{\frac{|\mu|h_{\frac{\mu}{|\mu|}}+(1-|\mu|)\delta_{\infty}}{|\mu|\lambda_{\frac{\mu}{|\mu|}}+(1-|\mu|)L}:\int \phi_i \diff \mu =\gamma_i, 1 \leq i \leq k   \right\}-\varepsilon. \]
and \(\int \phi_i \diff \mu=\gamma_i\) for all \(i \leq k\). If \(\mu \in \MCM(\Lambda,T)\) then we are finished, and \(|\mu|=0\) would imply that \(\ugamma=\underline{0}\), so assume that \(0<|\mu| <1\). For \(0<\varepsilon'< \varepsilon\), we may find \(\nu \in \MCM(\Lambda,T)\) such that \(h_{\nu}>\delta_{\infty}-\varepsilon\), \(|\lambda_{\nu}-L|<\varepsilon\), and \(|\int \phi_i \diff \nu|<\varepsilon'\) for all \(i \leq  k\). By our choice of \(\varepsilon\), there exists \(\mu' \in \MCM(\Lambda,T)\) such that
\[\int \phi_i \diff \mu'=\gamma_i-\frac{\varepsilon}{\varepsilon'}\int \phi_i \diff \nu, \: \forall i \leq k. \]
Let
\[\alpha:=\frac{1-|\mu|}{1-|\mu|+\frac{\varepsilon}{\varepsilon'}}.\]
Then the measure
\[\eta:=(1-\alpha) \mu+\alpha \mu'+(1-\alpha)(1-|\mu|) \nu \in \MCM(\Lambda,T) \]
satisfies \(\int \phi_i \diff \eta=\gamma_i\) for all \(i \leq  k\). If we let \(M=\sup\{\log|T'|\}\), then by the affinity of the entropy map we have
\begin{align*}
    \frac{h_{\eta}}{\lambda_{\eta}}& \geq \frac{(1-\alpha)|\mu|h_{\frac{\mu}{|\mu|}}+(1-\alpha)(1-|\mu|)(\delta_{\infty}-\varepsilon)}
    {(1-\alpha)|\mu|\lambda_{\frac{\mu}{|\mu|}}+\alpha M+(1-\alpha)(1-|\mu|)(L+\varepsilon)}.
\end{align*}
Moreover, since \(0<\varepsilon'<\varepsilon\) were arbitrary, \(\alpha\) and \(\varepsilon\) can be made arbitrarily small. This finishes the proof. 
\end{proof} 

We finish by giving an example where it does not suffice to take the supremum over probability measures. Consider the map \(F_{\lambda}\) as in the previous section and functions \(\phi_i=1_{I_{i+1}}\) for \(i \in \N\). Then \(\underline{0} \in Z_0\) since \(\Pi^* \delta_{(1,1,1\ldots)} \in \MCM(\Lambda,F_{\lambda})\) but \(\alpha_4(\underline{0})=\log 4>0=\alpha_3(\underline{0}).\) For a non-zero example, take \(0<\varepsilon<1/2\) and note that \((\frac{\varepsilon}{2},\frac{\varepsilon}{4},\frac{\varepsilon}{8},\ldots) \in Z_0\) since \((1-\varepsilon)\Pi^* \delta_{(1,1,1\ldots)}+\varepsilon \sum_{i=1}^{\infty}\frac{\Pi^* \delta_{(i,i,i\ldots)} }{2^i} \in \MCM(\Lambda,T) \), but \(\alpha_3(\frac{\varepsilon}{2},\frac{\varepsilon}{4},\frac{\varepsilon}{8},\ldots) \leq \varepsilon \log 4<(1-\varepsilon) \log 4 \leq \alpha_4(\frac{\varepsilon}{2},\frac{\varepsilon}{4},\frac{\varepsilon}{8},\ldots) \).


\begin{thebibliography}{}
\bibliographystyle{alphanum} 

\bibitem[BI06]{BI06}
L. Barriera, G. Iommi \textit{Suspension Flows Over Countable Markov Shifts}
Journal of Statistical Physics, Vol. 124, No. 1, July 2006,
\url{https://doi.org/10.1007/s10955-006-9140-9}

\bibitem[Bes34]{Bes34}
A.S. Besicovitch, \textit{On the sum of digits of real numbers represented in the dyadic system}.
Math. Annalen, 110 (1934), 321-30

\bibitem[BSa01]{BSa01}
L. Barreira and B. Saussol, \textit{Variational principles and mixed multifractal spectra},
Trans. Amer. Math. Soc. 353 (2001), no. 10, 3919–3944 (electronic). MR 1837214
(2002d:37048)

\bibitem[BSc00]{BSc00}
L. Barreira and J. Schmeling, \textit{Sets of “non-typical” points have full topological entropy and full Hausdorff dimension}, Israel J. Math. 116 (2000), 29–70. MR 1759398
(2002d:37040)

\bibitem[BSS02a]{BSS02a}
L. Barreira, B. Saussol, and J. Schmeling, \textit{Distribution of frequencies of digits via
multifractal analysis}, J. Number Theory 97 (2002), no. 2, 410–438. MR 1942968
(2003m:11124)

\bibitem[BSS02b]{BSS02b}
L. Barreira, B. Saussol, and J. Schmeling, \textit{Higher-dimensional multifractal analysis}, J. Math. Pures Appl. (9) 81 (2002), no. 1, 67–91. MR 1994883 (2004g:37038)

\bibitem[BT12]{BT12}
H. Bruin and M. Todd, \textit{Transience and thermodynamic formalism for infinitely branched
interval maps}. J. London Math. Soc. 86 (2012), 171–194

\bibitem[BT15]{BT15}
H. Bruin and M. Todd. \textit{Wild attractors and thermodynamic formalism}. Monatshefte für Mathematik. 2015 Sep;178(1):39-83 . \url{https://doi.org/10.1007/s00605-015-0747-2}

\bibitem[Caj81]{Caj81}
H. Cajar, \textit{Billingsley dimension in probability spaces}, Lecture Notes in Mathematics,
vol. 892, Springer-Verlag, Berlin, 1981. MR 654147 (84a:10055)

\bibitem[Dur97]{Dur97}
A. Durner, \textit{Distribution measures and Hausdorff dimensions, Forum Math}. 9 (1997),
no. 6, 687–705. MR 1480551 (98i:11060)

\bibitem[Egg49]{Egg49}
H.G. Eggleston, \textit{The fractional dimension of a set defined by decimal properties}, The Quarterly Journal of Mathematics, Volume os-20, Issue 1, 1949, Pages 31–36, \url{https://doi.org/10.1093/qmath/os-20.1.31}

\bibitem[FF00]{FF00}
A.-H. Fan and D.-J. Feng, \textit{On the distribution of long-term time averages on symbolic
space}, J. Statist. Phys. 99 (2000), no. 3-4, 813–856. MR 1766907 (2002d:82003) 

\bibitem[FFW01]{FFW01}
A. Fan, D. Feng and J. Wu, \textit{Recurrence, dimension and entropy}. Journal of the London Mathematical Society, Volume 64, Issue 1, August 2001, pp. 229 - 244, \url{https://doi.org/10.1017/s0024610701002137}

\bibitem[FLM10]{FLM10}
A. Fan, L. Liao and J. Ma, \textit{On the frequency of partial quotients of regular continued fractions}. Mathematical Proceedings of the Cambridge Philosophical Society, 148(1), 179-192. \url{doi:10.1017/S0305004109990235}

\bibitem[FLW02]{FLW02}
 D.-J. Feng, K.-S. Lau and J. Wu, \textit{Ergodic limits on the conformal repellers}, Adv.
Math. 169 (2002), no. 1, 58–91. MR 1916371 (2003j:37036)

\bibitem[FJLR15]{FJLR15}
A. Fan , T. Jordan, L. Liao and M. Rams, \textit{Multifractal analysis
for expanding interval maps with infinitely many branches}. Transactions of
the American Mathematical Society, 367 (no. 3), 1847-1870  (2015)

\bibitem[GR09]{GR09}
K. Gelfert and M. Rams, \textit{The Lyapunov spectrum of some parabolic systems}. Ergodic Theory and Dynamical Systems, 29 (2009), no. 1, 919-940   \url{doi:10.1017/S0143385708080462}

\bibitem[Gur69]{Gur69}
B. Gurevic, \textit{Topological entropy of enumerable Markov chains}. Dokl. Akad. Nauk SSSR,
187 (1969), 715-718; English translation: Soviet Math. Dokl., 10(4) (1969), 911–15

\bibitem[Iom05]{Iom05}
G. Iommi, \textit{Multifractal analysis for countable Markov shifts}. Ergodic Theory Dynam.
Systems 25 (2005), no. 6, 1881–1907

\bibitem[IJT15]{IJT15}
G. Iommi, T. Jordan and M. Todd, \textit{Recurrence and transience for suspension flows}. Israel J.
Math. 209 (2015), no. 2, 547–592

\bibitem[IJT17]{IJT17}
G. Iommi,T. Jordan and M. Todd, \textit{Transience and multifractal analysis.} Annales de l'Institut Henri Poincaré (C) Non Linear Analysis, 34(2), 407-421, \url{https://doi.org/10.1016/j.anihpc.2015.12.007}

\bibitem[ITV18]{ITV18}
 G. Iommi, M. Todd and A. Velozo, \textit{Upper semi-continuity of entropy in non-compact settings}, {arXiv:1809.10022v2}

\bibitem[ITV19]{ITV19}
 G. Iommi, M. Todd and A. Velozo, \textit{Escape of entropy for countable Markov shifts}, arXiv:1908.10741v1
 
\bibitem[IV19]{IV19}
G. Iommi and A. Velozo, \textit{The space of invariant measures for countable Markov shifts}. Journal d'Analyse Mathématique. Volume 143, pages 461-501 (2021)
 
\bibitem[JMU06]{JMU06}
 O. Jenkinson, R. D. Mauldin, and M. Urba\'nski, \textit{Zero temperature limits of Gibbs equilibrium states for countable alphabet subshifts of finite type}, J. Stat. Phys. 119
(2005), no. 3-4, 765–776. MR 2151222 (2006g:37051) 

\bibitem[Kni34]{Kni34}
V. Knichal, \textit{Dyadische Entwicklungen und Hausdorffsches Mass}. Časopis pro pěstování matematiky a fysiky 065.4 (1936): 195-210. \url{http://eudml.org/doc/20173}

\bibitem[Oli98]{Oli98}
E. Olivier, \textit{Analyse multifractale de fonctions continues}, C. R. Acad. Sci. Paris S\'er. I
Math. 326 (1998), no. 10, 1171–1174. MR 1650242 (99h:58109)

\bibitem[Oli99]{Oli99}
E. Olivier, \textit{Dimension de Billingsley d’ensembles satur\'es}, C. R. Acad. Sci. Paris S\'er. I
Math. 328 (1999), no. 1, 13–16. MR 1674409 (2000b:28021)


\bibitem[Oli00]{Oli00}
E. Olivier, \textit{Structure multifractale d’une dynamique non expansive d\'efinie sur un ensemble de Cantor}, C. R. Acad. Sci. Paris S\'er. I Math. 331 (2000), no. 8, 605–610.
MR 1799097 (2002g:37033)

\bibitem[Ols02]{Ols02}
 L. Olsen, \textit{Divergence points of deformed empirical measures}, Math. Res. Lett. 9
(2002), no. 5-6, 701–713. MR 1906072 (2003k:37038)


\bibitem[Ols03a]{Ols03a}
L. Olsen, \textit{Multifractal analysis of divergence points of deformed measure theoretical Birkhoff averages}, J. Math. Pures Appl. (9) 82 (2003), no. 12, 1591–1649. MR 2025314
(2004k:37036)


\bibitem[Ols03b]{Ols03b}
L. Olsen, \textit{Small sets of divergence points are dimensionless}, Monatsh. Math. 140
(2003), no. 4, 335–350. MR 2026104 (2005a:37034)

\bibitem[OW03]{OW03}
L. Olsen and S. Winter, \textit{Normal and non-normal points of self-similar sets and divergence points of self-similar measures}, J. London Math. Soc. (2) 67 (2003), no. 1,
103–122. MR 1942414 (2003i:28009)

\bibitem[OW07]{OW07}
L. Olsen and S. Winter, \textit{Multifractal analysis of divergence points of deformed measure theoretical Birkhoff averages}. II. Non-linearity, divergence points and Banach space valued spectra, Bull. Sci. Math. 131 (2007), no. 6, 518–558. MR 2351308 (2010b:28023)

\bibitem[PS07]{PS07}
C. E. Pfister and W. G. Sullivan, \textit{On the topological entropy of saturated sets}, Ergodic
Theory Dynam. Systems 27 (2007), no. 3, 929–956. MR 2322186 (2008f:37036)

\bibitem[PU10]{PU10}
F. Przytycki and M. Urba\'nski, \textit{Conformal fractals: ergodic theory methods}, London Mathematical Society Lecture Note Series, vol. 371, Cambridge University Press,
Cambridge, 2010. MR 2656475 (2011g:37002)

\bibitem[PW01]{PW01}
Y. Pesin and H. Weiss, \textit{The multifractal analysis of Birkhoff averages and large deviations}, Global analysis of dynamical systems, Inst. Phys., Bristol, 2001, pp. 419–431.
MR 1858487 (2002m:37034)

\bibitem[Sar99]{Sar99}
O. Sarig, \textit{Thermodynamic formalism for countable Markov shifts}. Ergodic Theory Dynam.
Systems 19 (1999), 1565–1593

\bibitem[Sar15]{Sar15}
O. Sarig, \textit{Thermodynamic formalism for countable Markov shifts}. Hyperbolic dynamics,
fluctuations and large deviations, 81–117, Proc. Sympos. Pure Math. 89, Amer. Math.
Soc., 2015

\bibitem[SV97]{SV97}
B. Stratmann and R. Vogt, \textit{Fractal dimension of dissipative sets}, Nonlinearity 10 (1997) 565–577

\bibitem[Tem01]{Tem01}
 A. A. Tempelman, \textit{Multifractal analysis of ergodic averages: a generalization of Eggleston’s theorem}, J. Dynam. Control Systems 7 (2001), no. 4, 535–551. MR 1854035
(2002g:37008)

\bibitem[Vol58]{Vol58}
B. Volkmann, \textit{\"Uber Hausdorffsche Dimensionen von Mengen, die durch Ziffern eigenschaften charakterisiert sind. VI}, Math. Z. 68 (1958), 439–449. MR 0100578 (20
\#7008)

 \bibitem[Wal81]{Wal81}
 P. Walters, \textit{An introduction to ergodic theory}, Graduate Texts in Mathematics 79, Springer,
1981

\end{thebibliography}
\end{document}